\documentclass{amsart}

\usepackage{amsthm} 
\usepackage{color}
\usepackage{mathptmx, wrapfig}
\usepackage{amsthm} 
\usepackage{color}

\usepackage{tikz}
\usetikzlibrary{matrix}

\usepackage{graphicx,epsfig}
\usepackage{amsmath, amssymb, latexsym, euscript,amscd}
\usepackage{url}
\usepackage[all]{xy}

\usepackage{exscale}

\usepackage{graphicx,epsfig}
\usepackage{epstopdf}
\usepackage{amsmath, amssymb, latexsym, euscript,amscd}
\usepackage{url}
\usepackage[all]{xy}
\usepackage{psfrag}
\usepackage{mathrsfs}

\usepackage[colorlinks,citecolor=blue,linkcolor=blue, backref=page]{hyperref}
\usepackage[nameinlink]{cleveref}
\usepackage{nameref}

\newcounter{notes}%

\definecolor{darkgreen}{rgb}{0.0, 0.5, 0.0}

\swapnumbers
\newtheorem{theorem}{Theorem}[section]
\newtheorem{lemma}[theorem]{Lemma}
\newtheorem{corollary}[theorem]{Corollary} 
\newtheorem{definition}[theorem]{Definition} 
\newtheorem{proposition}[theorem]{Proposition}

\newtheorem{open}{Open Theorem}
\newtheorem{closed}{Closed Theorem}
\newtheorem{clopen}{Clopen Theorem}

\def\Rep{\operatorname{Rep}}

\def\gap{\vspace{.3cm}\noindent}

\def\smallskip{\vspace{.15cm}}
\def\medskip{\vspace{.3cm}}
\def\text{\mbox}
\def\rh2{{\mathbb R}{\mathbb H}^2}
\def\ch2{{\mathbb C}{\mathbb H}^2}
\def\RR{{\mathbb R}}

\def\O{\operatorname{O}}
\def\ZZ{{\mathbb Z}}
\def\HH{{\mathbb H}}

\def\PP{\operatorname{{\mathbb P}}}
\def\R{{\mathbb R}}
\def\SS{\mathbb R\operatorname{P}_+}

\def\Fr{\operatorname{Fr}}

\def\interior{\operatorname{int}}
\def\SL{\operatorname{SL}}

\def\SO{\operatorname{SO}}
\def\PGL{\operatorname{PGL}}
\def\GL{\operatorname{GL}}

\def\Aff{\operatorname{Aff}}
\def\Hom{\operatorname{Hom}}
\def\H2R{{\mathbb H}^2\times {\mathbb R}}

\def\cl{\operatorname{cl}}
\def\dev{\operatorname{dev}}

\def\Aff{\operatorname{Aff}}

\def\C2{\operatorname{C^2}}
\def\halfgap{\vspace{.05in}}

\def\Vcal{\mathcal V}

\def\Ccal{\mathcal C}
\def\Bcal{\mathcal B}

\def\Dcal{\mathcal P}

\def\Dcal{\mathcal D}

\def\Tcal{\mathcal T}

\def\boxset{\mathfrak B}

\def\CH{\operatorname{CH}}
\def\PO{\operatorname{PO}}
\def\bdy{\partial}
\def\RP{\operatorname{\mathbb RP}}
\def\vol{\operatorname{vol}}

\def\cm{\hat{\mu}}
\def\Diag{\operatorname{Diag}}
\def\Wcal{\mathcal W}

\def\length{\operatorname{length}}

\def\RPn{{\mathbb{RP}}^n}

\def\Aut{\operatorname{Aut}}

\definecolor{back}{RGB}{255,255,255}
\definecolor{fore}{RGB}{0,0,0}
\definecolor{title}{RGB}{255,0,90}

\definecolor{green}{rgb}{0.0, 0.5, 0.0}
\definecolor{purple}{rgb}{0.5, 0.0, 0.5}
\definecolor{bluegreen}{rgb}{0.0,0.5, 0.5}
\definecolor{orange}{rgb}{1,0.5, 0.1}
\definecolor{redgreen}{rgb}{0.5, 0.5, 0.0}

\def\green{\color{green}}

\def\dvol{\operatorname{dvol}}
\def\Id{\operatorname{I}}
\def\FG{\operatorname{RA}}

\def\green{\color{green}}

\def\g2{{\green 2}}

\newcommand{\bv}{\left[\begin{array}{c}}
\newcommand{\ev}{\end{array}\right]}
\newcommand{\bbmat}{\begin{bmatrix}} 
\newcommand{\ebmat}{\end{bmatrix}}
\newcommand{\bmat}{\begin{matrix}} 
\newcommand{\emat}{\end{matrix}}
\newcommand{\bpmat}{\begin{pmatrix}} 
\newcommand{\epmat}{\end{pmatrix}}

\begin{document}
\title{The space of strictly convex real-projective structures on a closed manifold}
\date{\today}

\author{Daryl Cooper }
\email{cooper@math.ucsb.edu}

\address{Department of Mathematics, University of California, Santa Barbara, CA 93106, UlSA}

\author{Stephan Tillmann}
\email{stephan.tillmann@sydney.edu.au}
\address{School of Mathematics and Statistics, The University of Sydney, NSW 2006, Australia}


\begin{abstract} We give a new proof of the fact that, if $M$ is a compact $n$--manifold with no boundary, then the set
of holonomies of strictly convex real-projective structures on $M$ is a subset
of  $\Hom(\pi_1M,\PGL(n+1,\mathbb{R}))$ that is both open and closed.
\end{abstract}
\maketitle

 A {\em properly convex} domain is a convex open set  in real projective space whose closure contains no projective line. 
The quotient of such a domain by a discrete torsion-free
group of projective transformations that preserve the domain is a properly convex manifold.
 If, in addition, the  frontier of the domain contains no segment of a projective line
 bigger than one point, then the manifold is {\em strictly convex}. A hyperbolic manifold is a strictly convex projective manifold. 
 
 The case of dimension $2$ had been much studied. This is the simplest case of higher Teichm\"uller theory. A basic result is that 
 a strictly convex structure  on a closed manifold is uniquely determined by the conjugacy class of  the holonomy homomophism into the projective general linear group.
  Moreover, if the dimension of the manifold is at least $2$, then the holonomies that arise are precisely those in certain topological components of the representation variety. The goal of this paper
  is to provide self contained proofs of these results. 
  
  If $G$ is a Lie group, and $X$ is a manifold on which $G$ acts analytically and transitively, then a $(G,X)$--structure on a manifold is a maximal collection
  of charts which take values in $X$ and such that the transition functions are in $G$. For such structures the {\em Ehresmann-Thurston principle}
  says that  when $M$ is {\em closed} (compact without boundary),  then the set of holonomies of $(G,X)$--structures is an open subset of the representation variety.
  However properly convex structures do not fit into this scheme. The basic issue is that the convex open set in projective space varies. The set of such holonomies is
  often not closed. For example, the space of hyperbolic structures on the circle is not closed in $\operatorname{PO}(1,1)$.  
  In general, the space of all projective structures 
  on a manifold of dimension larger than $2$ is mysterious.

If $M$ is a compact $n$--manifold, then $\Rep(M)=\Hom(\pi_1M,\PGL(n+1,\RR))$ is a real algebraic variety. Let
$\Rep_{P}(M)$ and $\Rep_{S}(M)$ be, respectively, the subsets of $\Rep(M)$ of holonomies of properly convex, and of strictly convex structures  on $M$.
Then $\Rep_S(M)\subseteq\Rep_P(M)$. Throughout this paper we use the Euclidean topology everywhere, and not the {\em Zariski topology}.
In the following, $M$ is closed and $n\ge 2$.

\begin{open}\label{open} 
$\Rep_{P}(M)$ is open in $\Rep(M)$.
  \end{open}
  
\setcounter{closed}{1}  
\begin{closed}\label{closed} 
$\Rep_{S}(M)$ is closed in $\Rep(M)$.  
 \end{closed}

\setcounter{clopen}{2}  
\begin{clopen}
$\Rep_S(M)$ is a union of connected components
of $\Rep(M)$.
\end{clopen}

It follows from (\ref{unique}) that the holonomy of a strictly convex structure uniquely determines a projective manifold up to projective isomorphism.
The Open Theorem is due to Koszul~\cite{Kos1, Kos2}. Our proof is distilled from his. There is an affine manifold that is the tautological line bundle over a projective manifold, $M$.  Then $M$ is properly convex
if and only if there is an outwards-convex section of this bundle such that every connected flat subset of the section is contained in a simplex.
   This type of convexity  is easily shown to be preserved by small deformations. 

The Closed Theorem when $n=2$ is due to Choi and Goldman~\cite{CW2, MR2170138}, to Kim~\cite{Kim} when $n=3$, and to Benoist~\cite{Ben5} in general.
Our proof is new, and based on a geometric argument called {\em the box estimate} (\ref{boxestimate}). Given a sequence
of strictly convex projective structures whose holonomies converge, the associated domains might degenerate.  Benzecri's compactness theorem \cite{Benz}
implies one may apply projective transformations so that the domains converge to a properly convex domain, but then the holonomies might
diverge. The box estimate implies that if one chooses the conjugacies carefully, then the holonomies also converge.
This involves an elementary geometric fact (\ref{centerexists}) about an analogue of centroids for subsets of the sphere.

The Clopen Theorem follows from  the Open Theorem and the fact,  due to Benoist,  
 that a properly convex manifold that is homeomorphic to a strictly
convex manifold is also strictly convex (\ref{stricthomeo}). 
We have made an effort to make the paper  self-contained by including  proofs
of all the foundational results needed in \Cref{background,sec:cones}. 
We have endeavoured to simplify these proofs as much as possible.
See \cite{MR2464391} for an excellent survey.

There are extensions of the Open and Closed Theorems
when $M$ is the interior of a compact manifold with boundary, see \cite{CLT2} and forthcoming work. 
Indeed, the techniques in this paper were developed to handle this more general situation for which the pre-existing
methods do not suffice.
But it seems useful to present some of the main ideas
in the simplest setting. 

 The first author thanks the Sydney Mathematical Research Institute (SMRI) for partial support and hospitality while writing this paper. He also thanks the audience at SMRI that participated in a  presentation
of this material.
  Research of the second author is supported in part under the Australian Research Council's ARC Future Fellowship FT170100316. We thank the referee for a splendid job and a shorter proof of (\ref{nilpotent}).

\section{Properly and strictly convex}\label{background}

This section and the next review some well known results concerning properly and strictly convex projective manifolds that are needed to prove the theorems. 
The results needed subsequently are   (\ref{nilpotent})--(\ref{stricthomeo}) at the end of this section.
The only mild innovation
is that, to avoid appealing to results about word-hyperbolic groups, a more direct approach was taken to the proof of (\ref{stricthomeo}).
 The early sections of \cite{CLT1}  greatly expand on the background in this section.

 If $M$ is a manifold, the universal cover is $\pi_{M}:\widetilde M \to M$, and if $g\in\pi_1M$,  then $\tau_g:\widetilde M\to\widetilde M$
 is the covering transformation  corresponding to $g$. A {\em geometry} is a pair $(G,X)$, where $G$ is a group that acts analytically and transitively 
 on a manifold $X$.
 A $(G,X)$--structure on a manifold $M$ is determined by a {\em development pair} $(\dev,\rho)$ that consists of the {\em holonomy} $\rho\in \Hom(\pi_1M,G)$ and 
 the {\em developing map} $\dev:\widetilde M\to X$ which  is a local homeomorphism. The pair
 satisifies for all $x\in\widetilde M$ 
and $g\in\pi_1M$ that $\dev(\tau_{g}x)=(\rho g)\dev x$.

 
In what follows $V=\RR^{n+1}$, its dual vector space is $V^*=\Hom(V,\RR)$, and $\RR^{n+1}_0=V_0=V\setminus \{0\}$.
{\em Projective space} is $\PP V=V_0/\RR_0$, and $[A]\in\Aut(\PP V)=\PGL(V)$ acts on $\PP V$ by $[A][x]=[Ax]$.
{\em Projective geometry}  
 is $(\Aut(\PP V),\PP V)$ and is also written $(\Aut(\RP^n),\RP^n)$. {\em Radiant Affine geometry} is $(\GL(V),V_0)$.

We use the notation $\RR^+=(0,\infty)$. {\em Positive projective space} is $\RP^n_+=\PP_+V=V_0/\RR^+$ and 
$[x]_+=\{\lambda x:\lambda>0\}$ for $x\in V_0$. Sometimes we identify $\PP_+V$ with the unit sphere $S^n\subset\RR^{n+1}$
via $[x]_+\equiv x/\|x\|$.
The group $\Aut(\PP_+V):=\SL_{\pm}(V)\subset\GL(V)$ is the subgroup  with $\det=\pm1$. 
{\em Positive projective geometry} $(\Aut(\PP_+V),\PP_+V)$ is the double cover of projective geometry. 
We will pass back and forth between projective geometry and positive projective geometry without mention, often omitting the term {\em positive}. 

 If $U$ is a vector subspace of $V$, then $\PP U$ is a {\em projective subspace} of $\PP V$. This
is a {\em (projective) line} if $\dim U=2$ and a {\em (projective) hyperplane} if $\dim U=\dim V-1$.
The {\em (projective) dual} of $U$ is the projective subspace $\PP U^0\subset\PP V^*$ where  $U^0=\{\phi\in V^* :\phi(U)=0\}$.
We use the same terminology in positive projective geometry.
By lifting developing maps one obtains:

\begin{proposition}\label{holonomylifts} Every projective structure on $M$ lifts to a
positive projective structure. 
\end{proposition}
 
The {\em frontier} of a subset $X\subset Y$ is $\Fr X=\cl(X)\setminus\interior(X)$, and the {\em boundary} is $\bdy X=X\cap\Fr X$. 
A {\em segment} is a  connected, proper subset of a projective  line that contains more than one point.
In what follows $\Omega\subset\RP^n$. If $H$ is a hyperplane and $x\in H\cap\Fr\Omega$ and $H\cap\interior\Omega=\emptyset$, then $H$ is  called a {\em supporting
hyperplane (to $\Omega$) at $x$}. 
The set $\Omega$
\begin{itemize}
\item is {\em convex} if every pair of points in $\Omega$ is contained in a segment in $\Omega$.
\item is {\em properly convex} if it is convex,  and $\cl\Omega$ does not contain a projective line.
\item  is {\em strictly convex} if it is properly convex and $\Fr\Omega$
does not contain a segment. 
\item is {\em flat} if it is convex and $\dim\Omega<n$.
\item is a {\em convex domain} if it is a properly convex open set.
\item is $C^1$ if  for each $x\in\Fr\Omega$
there is a unique supporting hyperplane  at $x$.
\end{itemize}

The same terms will be applied to a lift of $\Omega$ to $\SS^n$.
If $V=U\oplus W$,  then {\em projection from $\PP W$ onto $\PP U$} is $\pi:\PP V\setminus\PP W\longrightarrow\PP U$
given by $\pi[u+w]=[u]$. {\em Duality} is the map that sends each point $\theta=[\phi]\in\PP V^*$ to the hyperplane
$H_{\theta}=\PP\ker\phi\subset\PP V$. If $L$ is a line in $\PP V^*$ the hyperplanes $H_{\theta}$ dual to the points 
$\theta\in L$
are called a {\em pencil of hyperplanes}. Then $Q=\cap H_{\theta}$ is the projective dual of $L$ and is called the {\em core of the pencil}.
It is a codimension $2$ projective subspace.

\begin{lemma} A convex set $\Omega$ is properly convex if and only if $\cl\Omega$ is disjoint from some hyperplane.
\end{lemma}
\begin{proof} Without loss of generality, suppose $\Omega$ is closed. The reverse implication is immediate since the complement of a hyperplane contains no line. 

Hence assume $\Omega$ is properly convex. The result is immediate if $n=1$, so assume $n>1$. Since $\Omega$ is convex and contains no projective line, it is simply connected, and so lifts to $\Omega'\subset\SS^n$.
Let $H\subset \SS^n$ be a hyperplane. Then $H\cap \Omega'$ is empty or properly convex. By induction
on dimension, $H$ contains a projective subspace $Q$  with $\dim Q=n-2$ that is disjoint from $H\cap \Omega'$. There is a pencil of hyperplanes $H_{\theta}$ with core $Q$. Now $\Omega'\cap H_{\theta}$ is contained in one of the two components of $H_{\theta}\setminus Q$. As $\theta$
moves half way around $\SS^1$ the component must change. Thus for some $\theta$ the intersection is empty.
\end{proof}

In what follows $\Omega\subset \RP^n$ is a properly convex domain, and $\Aut(\Omega)\subset\Aut(\RP^n)$ is the subgroup that preserves $\Omega$. The {\em Hilbert metric $d_{\Omega}$} on $\Omega$ is  defined as follows.
If $\ell\subset\RP^n$ is a line and $\alpha=\Omega\cap\ell\ne\emptyset$,  then $\alpha$ is a 
{\em proper segment in $\Omega$} and $\alpha=(a_-,a_+)$
with $a_{\pm}\in\Fr\Omega$. There is a projective isomorphism 
$f:\alpha\rightarrow\RR^+$
and
\[
d_{\Omega}(x,y)=\frac{1}{2}\left|\;\log \frac{f(x)}{f(y)}\;\right|
\] 
for $x,y\in\alpha$. It follows from the next result that $d_{\Omega}$ satisfies the triangle inequality, and thus is a metric. 
It
is a {\em Finsler metric}: given by a norm on the tangent space.
The factor of $1/2$ ensures that this gives the hyperbolic metric when $\Omega$ is a round disc. A {\em geodesic} in $\Omega$ is a  curve whose length is the distance between its endpoints.
It is immediate that $\Aut(\Omega)$ acts by isometries of $d_{\Omega}$. 

The interior of a triangle in the projective plane is a properly convex domain that is projectively equivalent
to the positive quadrant in an affine patch $\RR^2$. A metric ball in this domain with center $p$ is a hexagon. The vertices of the hexagon lie
on the lines through $p$ that contain a vertex of the triangle.
A smooth curve in the interior of a triangle  is a geodesic for the Hilbert metric if, and only if,   the tangent line at each point on the curve intersects the same pair
of sides of the triangle. Thus, in a properly convex domain, there may be many geodesics with the same endpoints. 

 Let $H_{\pm}$ be supporting hyperplanes for $\Omega$ at $a_{\pm}\in\Fr\Omega$, and 
 $P=H_+\cap H_-$.  {\em Projection from $P$ onto $\alpha$}  is the map $\pi:\Omega\to\alpha=(a_-,a_+)$ given by $\pi[u+v]=[v]$ 
 where $[u]\in P$ and $[v]\in\alpha$.
The choice of $H_{\pm}$ is unique only if $a_{\pm}$ are $C^1$ points. 
If $\Omega$ is strictly convex, then for each $x\in\Omega$ by (\ref{projlemma})(iii) there is a unique point $\Pi(x)\in\alpha$
closest to $x$, and $\Pi:\Omega\rightarrow\alpha$ is continuous and called the {\em nearest point retraction}.
In general $\pi \ne\Pi$.

\begin{figure}[h]
                       \centering
                 \includegraphics[scale=1.0]{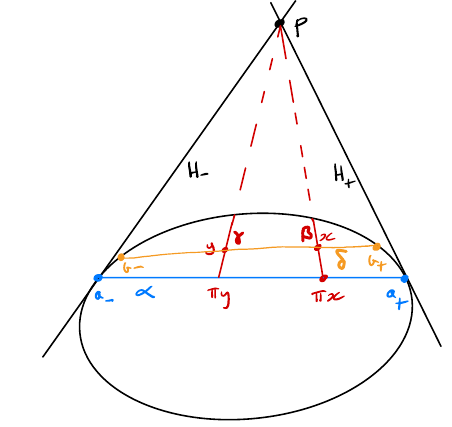}
 \caption{Projecting onto a line} \label{projection}
      \end{figure}

\begin{lemma}\label{projlemma} With the preceding notation\\ (i) $d_{\Omega}(x,y)\ge d_{\Omega}(\pi x,\pi y)$. \\
(ii) If $\Omega$ is strictly
convex,  then geodesics are segments of lines\\
(iii)  If $\Omega$ is strictly
convex, then  $\forall\ y\in\Omega\ \ \{x\in\Omega: d_{\Omega}(x,y)\le 1\}$ is strictly convex.
\end{lemma}
\begin{proof} \Cref{projection} represents  two {\em or more} dimensions.
Let $\delta=(b_-,b_+)$ be the proper segment in $\Omega$ containing $x$ and $y$, and $\pi:\Omega\rightarrow\alpha$ as above. Then $d_{\Omega}(x,y)=d_{\delta}(x,y)$. Since $(\pi|\delta):\delta\to\alpha$ is a projective embedding
$d_{\delta}(x,y)=d_{\pi\delta}(\pi x,\pi y)$. Also $d_{\pi\delta}(\pi x,\pi y)\ge d_{\alpha}(\pi x,\pi y)$ because 
$\pi\delta\subset\alpha$. Finally $d_{\alpha}(\pi x,\pi y)=d_{\Omega}(\pi x,\pi y)$, which proves (i).

Equality implies $\alpha=\pi\delta$. Thus, after relabelling if needed,  $b_{\pm}\in H_{\pm}$. Thus $[a_{\pm},b_{\pm}]\subset H_{\pm}\cap\Fr\Omega$.
If $\Omega$ is strictly convex it follows that $b_{\pm}=a_{\pm}$. This gives (ii). For (iii) see \cite{CLT1}(1.7).
\end{proof}

A {\em line} in a projective manifold is flat connected $1$-manifold. 
A  projective manifold $M$ is {\em convex} if
every pair of points in $M$ are contained in a line.  The developing map might not be injective on a lift of a line.
For example, every covering space of $\RP^1$ is convex.
It is easy to see (cf.\thinspace \Cref{fig:distance_btw_lines}) that:

\begin{proposition}\label{geodesiclimits}
Suppose that $\Omega$ is properly convex and $\gamma_1,\gamma_2:[0,1)\to\Omega$ are two lines  that are
both contained in the same projective plane, and
 $p_i=\lim_{t\to1}\gamma_i(t)\in\Fr\Omega$. Let $f(t)=d_{\Omega}(\gamma_1(t),\gamma_2)$. 
 \begin{enumerate}
\item If $p_1=p_2$, then $f$ is a non-increasing function. 
Moreover, $f(t)\to 0$ if and only if $\Omega$ is $C^1$ at $p_1$.
 \item If $p_1\ne p_2$, then $f$ is bounded if and only if  there is a segment in $\Fr\Omega$ that contains both $p_1$ and $p_2$ in its interior.
\end{enumerate}\end{proposition}

\begin{figure}[h]
                       \centering
\includegraphics[scale=0.9]{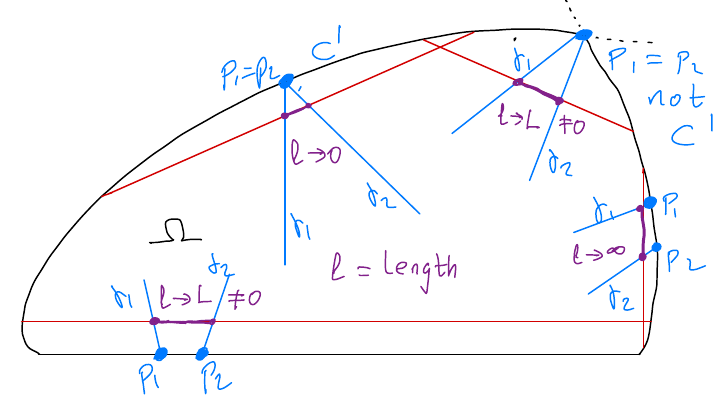}
 \caption{Distance between two lines near infinity}
 \label{fig:distance_btw_lines}
\end{figure}

A subset $\Ccal\subset V_0$ is a {\em cone} if $t\cdot\Ccal=\Ccal$ for all $t>0$.
  The {\em dual cone} $\Ccal^*=\interior\left(\{\phi\in V^*:\ \phi(\Ccal)\ge 0\}\right)$ is convex.  
  If $\Omega\subset\PP_+ V$,  then $\Ccal\Omega$ is the cone $\{v\in V_0:\ [v]_+\in\Omega\}$.
  If $\Omega$ is open and properly convex, then the {\em dual domain} 
  $\Omega^*=\PP(\Ccal\Omega^*)$ is open, and properly convex. Using the natural identification $V\cong V^{**}$
  it is immediate that $(\Omega^*)^*=\Omega$.
 \begin{corollary}\label{C1domain} If $\Omega$ is open and properly convex,  
 then $\Omega$ is $C^1$ (resp. strictly convex)
 if and only if $\Omega^*$ is strictly convex (resp. $C^1$).
 \end{corollary}
 \begin{proof} We have $[\phi]\in\Fr\Omega^*$ if and only if the dual hyperplane $H=\PP\ker\phi$ 
 supports $\Omega$. In this case we will always assume that $\phi(\Omega)\ge 0$.
If $p\in\Fr\Omega$, then there are
two distinct supporting hyperplanes $H_0=\PP\ker\phi_0$ and $H_1=\PP\ker\phi_1$ for $\Omega$ at $p$ 
if and only if
all the hyperplanes $H_t=\PP\ker \phi_t$ contain $p$ and support $\Omega$, where $\phi_t=(1-t)\phi_0+t\phi_1$ and $0\le t\le 1.$
This happens if and only if the segment
$\sigma=\{[\phi_t]:\ 0\le t\le 1\}$ is in $\Fr\Omega^*$. Hence $\Omega$ is $C^1$ if and only if $\Omega^*$ is strictly
convex. Replacing $\Omega$ by $\Omega^*$ and using $\Omega^{**}=\Omega$ gives the other result.
\end{proof}

 If $\Gamma\subset\Aut(\Omega)$
is  a discrete and torsion-free subgroup, then
 $M=\Omega/\Gamma$   is a {\em properly (resp. strictly) convex} manifold if $\Omega$ is properly (resp. strictly) convex.
Then $\widetilde M=\Omega$
and there is a development pair $(\dev,\rho)$, where $\dev$ is the identity map, and the holonomy $\rho:\pi_1M\to\Gamma$ is an isomorphism.

The {\em tautological line bundle} is the quotient map $\pi:V_0\rightarrow\PP V$. If $M$ is a projective manifold modelled on $\PP V$,  then the pull-back of this
line bundle by the developing map $\dev:\widetilde M\rightarrow \PP V$ is a line bundle over $\widetilde M$. Covering transformations induce
bundle isomorphisms, and the quotient is a line bundle over $M$ called  the {\em tautological line bundle}. It is a trivial bundle, since the lines
can be oriented to point way from $0$.

This bundle has a natural structure
as a radiant affine manifold. In the case that  $M=\Omega/\Gamma$ is properly convex,
the holonomy lifts to $\widetilde\Gamma\subset \SL_{\pm} V$. Then
this bundle over $\widetilde M$ is naturally identified with $\Ccal\Omega$ and
the quotient $\Ccal\Omega/\widetilde\Gamma$ is a radiant affine manifold (\ref{RAbundle}) that can be naturally identified with the tautological line bundle over $M$.
This plays a key role in the proof of the Open Theorem.
If $X$ is a subset of a metric space $Y$, the $r$-neighborhood of $X$ is 
\[N_r(X)=N(X,r)=\{y\in Y: d(y,X)\le r\}\]

We now use a compactness argument to show that if $\Omega$ is the universal cover of a closed strictly convex manifold, then projection 
from a projective subspace onto a geodesic in $\Omega$ is
uniformly distance decreasing in the following sense.
 
 \begin{lemma}\label{strictprojection} Suppose $M=\Omega/\Gamma$ is a strictly convex closed manifold. Given  $b>1$ there is  $R=R(b)>0$ so the following holds.
Suppose $\pi:\Omega\to\alpha$ is a projection  onto a geodesic $\alpha$.  Suppose that $\eta$ is a rectifiable arc in
 $\Omega\setminus N_R(\alpha)$
of length $\ell$ with endpoints $x$ and $y$. Then $d_{\Omega}(\pi x,\pi y)< 1+(\ell/b)$ \end{lemma}
 \begin{proof} By decomposing $\eta$ into  finitely many subarcs, with one of length at most $b$, and $\lfloor \ell/b \rfloor$ of length $b$,
 it suffices to prove if $\length(\eta)\le b$, then $d_{\Omega}(\pi x,\pi y)\le 1$.
  
 Write $d=d_{\Omega}$. If no such $R$ exists,  then there are sequences $x_k,y_k\in\Omega$ and projections $\pi_{_k}:\Omega\to\alpha_k$
 with $d(x_k,y_k)\le b$ and $d(x_k,\alpha_k)\ge k$ and $1\le d(\pi_{_k} x_k,\pi_{_k} y_k)\le b$. Then $\beta_k=[\pi_k x_k,x_k]$ and $\gamma_k=[\pi_{_k} y_k,y_k]$
 are geodesic segments and $\pi_{_k}\beta_k=\pi_{_k} x_k $ and $\pi_{_k}\gamma_k=\pi_{_k}y_k$. 
 
   There is a compact subset $W\subset\Omega$ such that $\Gamma\cdot W=\Omega$. By applying an element of $\Gamma$ we may assume that $\pi_{_k}x_k\in W$.
  After subsequencing we may assume $\beta=\lim \beta_k$, $\gamma=\lim \gamma_k$, $\alpha=\lim\alpha_k$ and $\pi=\lim \pi_{_k}$ all exist and $\pi:\Omega\to\alpha$ is
   projection. Refer to \Cref{projection}. Thus $\beta$ are $\gamma$ are geodesic rays in $\Omega$ that start on $\alpha$, and end on $\Fr\Omega$, with
  $\pi(\beta)=\alpha\cap\beta\ne \alpha\cap\gamma=\pi(\gamma)$.
  
 Since $\Omega$ is strictly convex, and $d(x_k,y_k)\le b$, it follows from (\ref{geodesiclimits}) that $x_k$ and $y_k$ 
  limit on the same point  $p\in \Fr\Omega$. Hence $p$ is the endpoint of both $\beta$ and $\gamma$ on $\Fr\Omega$. 
  Also $p\in P=H_-\cap H_+$, where $H_{\pm}$ are the supporting hyperplanes to $\Omega$ at the endpoints of $\alpha$.
   Since $\Omega$ is strictly convex, $p\notin P$. It follows that $\pi(\beta)=\pi(p)=\pi(\gamma)$ which is a contradiction.
   \end{proof}

 The following is due to Benoist \cite{MR2094116}.
    A {\em triangle}  is a disc $\Delta$ in a projective plane bounded by
 three segments.   If $\Delta\subset\cl\Omega$ and $\Delta\cap\Fr\Omega=\bdy\Delta$, then $\Delta$ is
  called a {\em properly embedded triangle} or {\em PET}.
  A properly convex set $\Omega$ has  {\em thin triangles} if there is $\delta>0$ such that for every triangle
 $T$ in $\Omega$, each side of $T$ is contained in a $\delta$-neighborhood of the union of the other two sides with respect to the Hilbert metric.
  
  \begin{proposition}\label{PETstrict} If $M=\Omega/\Gamma$ is a properly convex closed manifold, then the following are equivalent:
  \begin{enumerate}
\item  $M$ is strictly convex.
\item $\Omega$ does not contain a PET. 
\item $\Omega$ has thin triangles.
\end{enumerate}
 \end{proposition}
 \begin{proof}  $(ii)\Rightarrow(i):$ If $\Omega$ is not strictly convex, there is a maximal segment $\ell\subset\Fr\Omega$. Choose $x\in\Omega$ and let $P\subset\Omega$ 
be the interior of the triangle that is the convex hull of $x$ and $\ell$. Choose a sequence $x_n$ in $P$ that limits on the midpoint of $\ell$.
By (\ref{geodesiclimits}) $d_{\Omega}(x_n,\Fr P)\to\infty$ because $\ell$ is maximal. Since $M$ is compact there is compact set $W\subset\Omega$
such that $\Gamma\cdot W=\Omega$. Thus there is $\gamma_n\in\Gamma$ with $\gamma_nx_n\in W$. After
choosing a subsequence  $\gamma_nx_n\to y\in W$, and $\gamma_n P$ converges to the interior of a PET
 $\Delta\subset\Omega$ that contains $y$. 
 
 $(iii)\Rightarrow(ii):$ If $\Omega$ contains a PET $\Delta$,  then $d_{\Omega}|\Delta=d_{\Delta}.$ Now $\PGL(\Delta)$ contains a subgroup $G\cong\RR^2$ that acts transitively on $\Delta$.
Therefore $(\Delta,d_{\Delta})$ is isometric to a normed vector-space, thus does not
have thin triangles.

$(i)\Rightarrow(iii):$ If $\Omega$ does not have thin triangles, then there is a sequence  of triangles $T_k$ in $\Omega$
and points $x_k$ in $T_k$
with $d_{\Omega}(x_k,\bdy T_k)>k$. As above, after applying elements of $\Gamma$, we may assume $T_k$
converges to a PET, so $\Omega$ is not strictly convex.
 \end{proof}

It follows that in dimension $2$ either $\Omega$ is a triangle (and $M$ is a torus or Klein bottle) or else is strictly convex. Using a basic fact from geometric group theory,  it follows that a properly convex manifold (of any dimension) is strictly convex if and only if $\pi_1M$ has thin triangles.
This is an algebraic characterization
 of strictly convex. However we will give a direct geometric argument.
 
 \gap

   Given $K\ge 1$ and $L\ge 0$, a {\em $(K,L)$--quasi-isometric embedding}  is a map $f:X\to Y$ between metric spaces $(X,d_X)$ and $(Y,d_Y)$ such that
  $$K^{-1}(d_X(x,x')-L)\;\le\; d_Y(fx,fx')\;\le\; K\;d_X(x,x')+L $$
for all $x,x'\in X.$
     If $X=[a,b]\subset \RR$, then $f$ is a {\em quasi-geodesic}.    
  The map $f$ is a {\em quasi-isometry} or {\em QI} if also $Y\subset N(fX,L)$. 
 
  If $g:Y\to Z$ is a $(K',L')$--QI, then $g\circ f$ is a $(KK',K'L+L')$--QI.   
           In particular, if $g$ is a QI 
   and $f$ is a quasi-geodesic, then $g\circ f$ is a quasi-geodesic
   with QI constants that only depend on
those of $f$ and $g$.

  If $G$ is a finitely generated group, then  a choice of finite generating set gives a {\em word metric} on $G$.
   A different choice of generating set gives a bi-Lipschitz equivalent metric. The \v{S}varc-Milnor lemma in this setting is

  \begin{proposition}\label{QIhomeo} If $M=\Omega/\Gamma$ is a closed and properly convex manifold,
 then $(\Omega,d_{\Omega})$ is QI to $\pi_1M$. \end{proposition}
 \begin{proof} Fix $x\in\widetilde M=\Omega$. Then $f:\pi_1M\rightarrow\widetilde M$ given by $f(g)=\tau_g(x)$ is a QI.
  \end{proof}
 
   A metric space $(Y,d_Y)$ is {\em ML} if it satisfies the  {\em Morse Lemma}, i.e.\thinspace for all $(K,L)$ there is $S=S(K,L)>0$, called a {\em tracking constant}, such that if $\alpha$ and $\beta$ are $(K,L)$--quasi-geodesics in $Y$
  with the same endpoints,  then  $\alpha\subset N_{S}(\beta)$ and $\beta\subset N_{S}(\alpha)$. 
  Since quasi-geodesics are sent to quasi-geodesics by quasi-isometries, it follows that the property of ML is preserved by quasi-isometry.

  Clearly $\RR^2$ is not ML, and any norm on $\RR^2$ is QI to the standard norm, so $\RR^2$ with any norm is not ML.
  A geodesic in the Cayley graph of  $G$ picks out a sequence in $G$ that is a quasi-geodesic.
 It follows from (\ref{QIhomeo}) that whether or not a closed properly convex manifold is ML only depends on $\pi_1M$.

 \begin{proposition}\label{stricthyperbolic} If $M=\Omega/\Gamma$ is a properly convex closed manifold, then $M$ is strictly convex if and only if $(\Omega,d_{\Omega})$ is {\em ML}.
 \end{proposition}
 \begin{proof} If $\Omega$ is not strictly convex, then by (\ref{PETstrict}) it contains a PET $\Delta$. The metric $d_{\Omega}$ restricted to $\Delta$ is isometric to a norm on $\RR^2$. Hence $(\Omega,d_{\Omega})$ is not {\em ML}.

 Now suppose $\Omega$ is strictly convex. A quasi-geodesic is {\em nice} if the image consists of
 finitely many line segments and each line segment is parameterized by arc length.
Every  $(K,L)$--quasi-geodesic $\alpha$ in $\Omega$
 is approximated by a nice $(K',L')$-quasi-geodesic $\overline{\alpha}$ in the sense that 
 $\alpha\subset N(\overline{\alpha},D)$ and $\overline{\alpha}\subset N(\alpha,D)$ and $(D, K',L')$ only depends on $(K,L)$. Thus it suffices to prove there is a tracking constant for nice quasi-geodesics. 
 
 \begin{figure}[h]
                       \centering\includegraphics[scale=1.5]{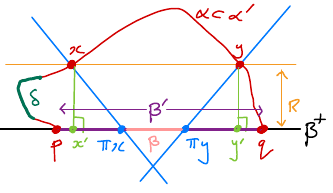}
 \caption{Quasi-geodesics} \label{quasigeodesic}
\end{figure}
 
Refer to \Cref{quasigeodesic}. 
 Let $\alpha'$ be a nice $(K,L)$--quasi-geodesic with endpoints $p$ and $q$.
  Let $\beta'$ be the line segment with the same endpoints as $\alpha'$.
  Let $R=R(2K)$ be given by (\ref{strictprojection}) and set
$S'=R+4RK+K+L$.
 We show that $S=2KS'+2L+2S'$ is a tracking constant.
 Let $\beta^+$ be the geodesic in $\Omega$ that contains $\beta'$ and $\pi:\Omega\rightarrow\beta^+$ be a projection.\\
 Claim 1) $\alpha'\subset N(\beta^+,S')$.\\
Claim 2) $\alpha'\subset N(\beta',S/2)$ and $\beta'\subset N(\alpha',S/2)$.

 Assuming the claims, if $\gamma'$ is another nice $(K,L)$--quasi-geodesic with endpoints $p$ and $q$,  then
  $$\beta'\subset N(\gamma',S/2)\qquad{\rm so}\qquad\alpha'\subset N(\beta',S/2)\subset N(N(\gamma',S/2),S/2)=N(\gamma',S)$$
 which proves that $S$ is a tracking constant. Hence $(\Omega,d_{\Omega})$ is ML.\gap

For Claim 1 suppose that $\alpha$ is the closure of a component of $\alpha'\setminus N(\beta^+,R)$. The endpoints
 $x$ and $y$ of $\alpha$ satisfy $d(x,\beta^+)=d(y,\beta^+)=R$. Let $x',y'\in\beta^+$ be chosen so that $d(x,x')=R=d(y,y')$ and 
 let $\beta=[\pi x,\pi y]\subset\beta^+$.  By (\ref{strictprojection}) $$\length(\beta)\le  1+\length(\alpha)/2K$$ by definition of $R$. 
Since $\pi$ is distance non-increasing $d(x',\pi x)=d(\pi x',\pi x)\le d(x', x)=R$, and similarly $d(y',\pi y)\le R$. Thus
 $$d(x,y)\le d(x,x')+d(x',\pi x)+d(\pi x,\pi y)+d(\pi y,y')+d(y',y)\le 4R+\length(\beta)$$
  Since $\alpha$ is parametrized
  by arc length, the domain of $\alpha$ is  $\length(\alpha)$, Since $\alpha$ is a 
  $(K,L)$--quasi-geodesic we have 
  $$\begin{array}{lrcl}
 & K^{-1}(\length(\alpha)-L)& \le & d(x,y)\\
 \Rightarrow & \length(\alpha)   &\le & K\cdot d(x,y)+L \\
&  & \le & K(4R +\length(\beta))+L\\
&  & \le & K(4R+1+\length(\alpha)/2K)+L\\
&  & = &4RK+K+L+\length(\alpha)/2\\
 \Rightarrow & (1/2)\length(\alpha)&\le& 4RK+K+L
 \end{array}$$
 Since the endpoints of $\alpha$ are in $N(\beta^+,R)$ it follows that
 $$\alpha\subset N(\ N(\beta^+,R)\ ,\ (1/2)\length(\alpha))=N(\beta^+,R+(1/2)\length(\alpha))$$ Now
 $$R+(1/2)\length(\alpha)\le R+4RK+K+L=S'$$
 This proves Claim 1.
  
For Claim 2, let $\Pi:\alpha'\rightarrow\beta^+$ be the nearest point retraction. This map is continuous, and $p,q$ are both in the image, so the image of
 $\Pi\circ\alpha'$
 contains  $\beta'$.  If $x'\in\beta'$ there is $y'\in\alpha'$ with $\Pi y'=x'$. Then $d(\beta^+,y')=d(x',y')$, and by Claim 1 $d(\beta^+,y')\le S'$ 
 so $d(x',y')\le S'$ thus
 $$ \beta'\subset N(\alpha',S')\subset N(\alpha',S/2)\qquad 
\&
\qquad\eta=\Pi^{-1}(\beta')\subset N(\beta',S')\subset N(\beta',S/2)$$
 Let $\delta\subset\alpha'$ be the closure of a component of
 $\alpha'\setminus\eta$, thus $\delta$ is an arc in the interior of $\alpha'$. Let $r,s$ be the endpoints of  $\delta$. 
 Then $\Pi(r)=\Pi(s)\in\{p,q\}$.   By Claim 1, 
 $$d(r,s)\le d(r,\Pi(r))+d(\Pi(r),\Pi(s))+d(\Pi(s),s)\le S'+0+S'=2S'$$
 Since $\delta$ is a  nice $(K,L)$--quasi-geodesic
  $$\length(\delta)\le K\cdot d(r,s) +L\le K\cdot 2S'+L$$ By Claim 1, the endpoints of $\delta$ are distance at most $S'$ from $\beta'$ so $$\delta\subset N(N(\beta',S'),(1/2)\length(\delta))\le N(\beta',S'+(1/2)(K\cdot 2S'+L))\subset N(\beta',S/2)$$
    Since $\delta\cup\eta\subset N(\beta',S/2)$ it follows that $\alpha'\subset N(\beta',S/2)$, proving Claim 2. \end{proof}
 
 
In the proof of the Closed Theorem, we construct a properly convex manifold with the correct holonomy. In order to know this manifold is closed and strictly convex we use the pair of results (\ref{htpyequiv}) and (\ref{stricthomeo}).
A properly convex manifold has contractible universal cover. Whitehead's theorem (that a weak homotopy equivalence between CW complexes is a homotopy equivalence) implies that if $M$ and $M'$ are properly convex and $\pi_1M\cong\pi_1M'$,  then $M$ and $M'$ are homotopy equivalent. This is key in the following result.

 \begin{lemma}\label{htpyequiv} Suppose $M$ and $M'$ are properly convex manifolds and
 $\pi_1M\cong\pi_1M'$. If $M$ is
closed, then $M'$ is closed.
\end{lemma}
\begin{proof} Let $\pi:\Omega\to M$ be the projection. This is the universal covering space of $M$, so $M$ is
a $K(\pi_1M,1)$. The same holds for $M'$. Hence $M$ and $M'$ are homotopy equivalent. If $n=\dim M$, then $M$ is closed if and only if 
$H_n(M;\ZZ_2)\cong\ZZ_2$. Since homology is an invariant of homotopy type the result follows.\end{proof}

\begin{corollary}\label{stricthomeo} Suppose $M$ and $N$ are closed and properly convex, and that $\pi_1M\cong\pi_1N$. If $M$ is strictly convex, then $N$ is strictly convex.
\end{corollary}
\begin{proof} Since ML is preserved by quasi-isometry, it follows from (\ref{stricthyperbolic}) and (\ref{QIhomeo}) that  $M$ is strictly convex if and only if  $\pi_1M$ is ML, and this is determined by $\pi_1M$.
\end{proof}


The remaining discussion in this sections comes logically after \Cref{sec:cones} since it depends on (\ref{dualcpct}), but it fits better into the narrative here. We let the reader choose their preferred order.

    The {\em displacement distance of $\gamma\in\Aut(\Omega)$} is
$t(\gamma)=\inf\{d_{\Omega}(x,\gamma x): x\in\Omega\}$. The element $\gamma$ is called {\em hyperbolic} if $t(\gamma)>0$. 
It is {\em elliptic} if it fixes a point in $\Omega$ and {\em parabolic} if it has no fixed point in $\Omega$ but $t(\gamma)=0$.
Proposition (2.1) in \cite{CLT1} says $t(\gamma)=\log|\lambda/\mu|$ where $\lambda$ and $\mu$ are the eigenvalues of 
$t(\gamma)$ of largest and smallest modulus.
We will not make use of elliptics or parabolics in this paper.

\begin{lemma}[Hyperbolics]\label{hyperbolic} Suppose $\Omega\subset \PP V$ and $M=\Omega/\Gamma$ 
is a strictly convex closed manifold and $1\ne \gamma\in\Gamma$. 
Then
 $\gamma$ is hyperbolic and there are
$a_{\pm}\in\Fr(\Omega)$ such that for all $x\in\PP V\setminus (H_+\cup H_-)$ we have
\[\lim_{n\to\pm\infty}\gamma^nx=a_{\pm}\]
where $H_{\pm}$ is the supporting hyperplane to $\Omega$ at $a_{\pm}$.
\end{lemma}
\begin{proof} Since $M$ is compact, the Arzela-Ascoli Theorem implies  there is a closed
 geodesic $C$ in $M$ that is conjugate to $\gamma$ in $\pi_1M$
and $t(\gamma)=\length(C)>0$. Hence $\gamma$ is hyperbolic. By (\ref{projlemma}) $C$ is covered by a proper segment $\alpha=(a_-,a_+)$ in $\Omega$
that is preserved by $\gamma$.
By (\ref{dualcpct}) $M$ is $C^1$ so there are unique supporting hyperplanes $H_{\pm}$ 
to $\cl\Omega$ that contain $a_{\pm}$ respectively. 
 Then $Q=H_+\cap H_-$
is a codimension--2 subspace, and it is disjoint from $\Omega$ by {\em strict} convexity. Moroever $Q$ is
 preserved by $\gamma$. 

 The pencil of hyperplanes $\{H_t:t\in L\}$ in $\PP V$ that contain $Q$ is dual to a line $L\subset\PP(V^*)$.
Since the dual of $\gamma$ acts on  $L\cong\RP^1$ projectively and non-trivially, it
 only fixes  the two points $[H_{\pm}]\in L$. The other hyperplanes $[H_t]$ are moved by $\gamma$ away from $[H_-]$ and towards $[H_+]$. Since $\Omega$ is strictly convex, $\gamma$ moves all points in $\Omega$  towards $a_+$. 
Suppose $\ell$ is a projective line that contains $a_-$. If $\ell$ is not contained
in $H_-$,  then, since $H_-$ is the unique supporting hyperplane at $a_-$, it follows that $\Omega\cap \ell\ne\emptyset$.
Thus $\gamma^k\ell\to\ell'$ where $\ell'$ is the projective line containing $\alpha$. Hence $\gamma^k(\ell\setminus a_-)\to a_+$
as $k\to\infty$.
  This reasoning applied to $\gamma^{-1}$ gives the corresponding statements for $a_-$, and gives the second conclusion.\end{proof}

The point $a_+$ is called the {\em attracting}, and $a_-$ is the {\em repelling}, fixed point of $\gamma$,
and $(a_-,a_+)\subset \Omega$ is called the {\em axis} of $\gamma$.  This axis is the only proper segment in $\Omega$ preserved by $\gamma$. The attracting fixed point of $\gamma^{-1}$ is the repelling fixed point of $\gamma$. The following result is used in our proof of Chuckrow's theorem, that is needed for the closed theorem.

\begin{corollary}[nilpotent subgroups]\label{nilpotent}  If $M=\Omega/\Gamma$ is a closed, strictly convex, projective manifold,  then 
every nilpotent subgroup of $\Gamma$ is cyclic.\end{corollary}
\begin{proof} By (\ref{hyperbolic}) every element of $\Gamma$ is hyperbolic, and the axis is the only segment preserved by a non-trivial hyperbolic in $\Gamma$. If $\Gamma'\subset\Gamma$ is nilpotent
then $\Gamma'$ has a non-trivial center. Let $\beta$ be a non-trivial element of this center, and let
$\ell$ be the axis of $\beta$. Since every element of $\Gamma'$ commutes with $\beta$
it follows that $\Gamma'$ preserves the axis of $\ell$.
The action of $\Gamma$ on $\Omega$
is free, so $\Gamma'$ acts freely on $\ell$. A discrete group acting freely by homeomorphisms on $\RR$ is cyclic.
\end{proof}

Two projective structures on the same manifold are {\em equivalent} if there is a projective isomorphism 
of one onto the other
that is {\em homotopic} to the identity (rather than {\em isotopic} to the identity). 
For closed {\em hyperbolic} manifolds of dimension at most $3$, homotopy implies isotopy \cite{MR0214087, MR1895350}.
The next two results imply that the holonomy of a strictly convex manifold determines the domain $\Omega$, and hence determines the projective structure, up
to {\em equivalence}. However, given the holonomy, there may be two developing maps that are homotopic but not isotopic.
Thus the space of equivalence classes of strictly convex projective structures on a closed manifold can be identified with a subset of 
the space of representations modulo conjugacy.

\begin{lemma}\label{noinvt} If $M=\Omega/\Gamma$ is a closed properly convex manifold and  
$\Omega'\subset\Omega$ is a non-empty properly convex subset
that is preserved by $\Gamma$, then $\Omega'=\Omega$.
\end{lemma}
\begin{proof} Otherwise the function $F:\Omega\to \RR$ given by $F(x)=d_{\Omega}(x,\Omega')$ is continuous, unbounded, and $\Gamma$--invariant.
Thus it covers a continuous unbounded function $f:M\to\RR$, contradicting the compactness of $M$.
\end{proof}

\begin{proposition}[Unique domain]\label{unique} Let $n\ge 2.$ Suppose $\Omega\subset\RP^n$ and $M=\Omega/\Gamma$ is a strictly convex, closed $n$--manifold.
If $\Omega'$ is open and properly convex and preserved by $\Gamma$, then $\Omega'=\Omega$.\end{proposition}
\begin{proof} 
 Let $X$ be union, over all $1\ne\gamma\in\Gamma$,  of attracting  fixed points of $\gamma$.
 Then $X\subset\Fr\Omega$ by (\ref{hyperbolic}).
Let $W$ be the  convex hull of $X$ in $\cl\Omega$.
Then $U=W\cap\Omega$ contains the axis of each hyperbolic in $\Gamma$, so $U$ is non-empty, convex, 
$\Gamma$-invariant. Since $U\subset\Omega$, we have $\Omega=U$ by (\ref{noinvt}).
Since $\Omega$ is strictly convex it follows $\Fr\Omega=\cl X$.

Since $\Omega'$ is preserved by $\Gamma$ it follows that $\cl\Omega'\supset \cl X$. Now $\Fr\Omega$ is a
convex hypersurface of dimension $n-1>0$. 
Only one side of $\Fr\Omega$ is locally convex.
Hence $\Omega'$ contains points on the same side of $\Fr\Omega$ as $\Omega$. Thus $U'=\Omega'\cap\Omega\ne\emptyset$ is
properly convex and
 is preserved by $\Gamma$.
 By (\ref{noinvt}) $\Omega=U'=\Omega'$.\end{proof}


\section{Convex Cones}\label{sec:cones}

This section is based on work of Vinberg, as simplified by Goldman. Write $V=\RR^{n+1}$. 
If $\Omega$ is properly convex there is a nice convex hypersurface  in the cone $\Ccal\Omega$
that is preserved by every projective automorphism of $\Ccal\Omega$. This is a generalization
of the hyperboloid model of hyperbolic space that sits inside the lightcone.
In general there is a choice of
such surfaces. One is Vinberg's {\em characteristic surface}, and another is an {\em affine sphere}.
The characteristic surface is defined below using a simple geometric condition. 
One consequence of this is (\ref{centerexists}),
which provides an analog of the centroid for properly convex subsets of the sphere.
The purpose of this section is to prove the following and derive some consequences.

\begin{theorem}\label{SPdef} Suppose $\Ccal$ is a  properly convex cone in $V$.  Let $\Dcal$ be the intersection of all closed affine halfspaces $H\subset V$
such that $\vol(\Ccal\setminus H)=1$. Then
\begin{enumerate}
\item[(1)] $\Dcal\subset\Ccal$.
\item[(2)] $\bdy\Dcal$ is a strictly convex hypersurface.
\item[(3)] $\bdy\Dcal$ meets every ray in $\Ccal$ once. 
\item[(4)] $\bdy\Dcal$ is preserved by $\SL_{\pm}(\Ccal)$.
\item[(5)] If $T$ is a supporting hyperplane to $\Dcal$,  then $T\cap\Dcal$ is the centroid of $T\cap\Ccal$ (see \Cref{Thetapicnew}).
\item[(6)] $\bdy \Dcal$ is $C^1$.
  \end{enumerate}
  \end{theorem}
  In the special case that $\Ccal$ is the cone on a round ball, then $\bdy \Dcal$ is the hyperboloid model of hyperbolic space.
Fix an inner product $\langle\;\cdot\;,\;\cdot\;\rangle$ on $V$. This determines a norm on $V$, and  induces a Riemannian metric
and associated volume form
on every smooth submanifold of $V$. 

 The {\em centroid} of a bounded convex set
$K$ in $V$ is the point $\mu(K)$ in $K$ given by
 $$\mu(K)=\left.\int_K x\ \dvol_{K_x}\right/\int_K \dvol_{K_x}$$
Here $\dim K\le \dim V$ and $\dvol_K$ is the induced volume form on $K$. The centroid is independent of the inner product.

\begin{figure}
                       \centering	 \includegraphics[scale=0.6]{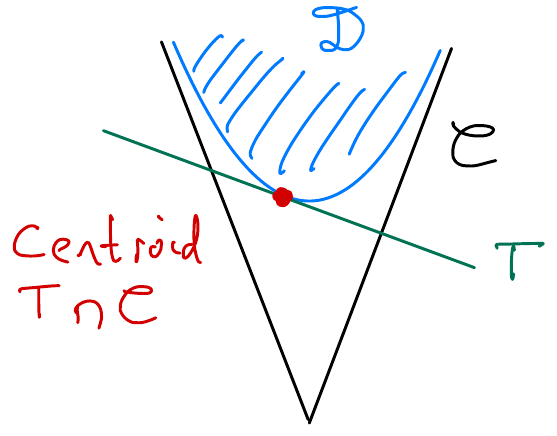}
 \caption{$T\cap \Dcal$ is the centroid of $T\cap\Ccal$} \label{Thetapicnew}
\end{figure}

Let $S^n=\{x\in V:\ \|x\|^2=1\}$.
Suppose $\Omega$ is an open, properly convex subset of $S^n$ and $\Ccal=\Ccal\Omega$ is the corresponding cone in $V$.
Given $0\ne \phi\in V^*$, the set $\phi^{-1}(1)$ is an affine hyperplane in $V$. If $\phi\in\Ccal^*$,  then $\phi(\Ccal)>0$, so
$\Ccal_{\phi}=\{x\in\Ccal: \phi(x)=1\}$ is a bounded subset of this hyperplane that separates $\Ccal$.
The subset of $\Ccal$ below this hyperplane is $\Ccal\cap\phi^{-1}(0,1]$, and has finite volume and  boundary $\Ccal_{\phi}$.
The centroid of $\Ccal_{\phi}$ is a point $\mu(\Ccal_{\phi})\in\Ccal$.
Define $\Theta(\phi)=\mu(\Ccal_{\phi})$. We will show this is a  homeomorphism $\Theta:\Ccal^*\rightarrow\Ccal$. 
 
 The {\em volume function} $\Vcal:\Ccal^*\to\RR$  is defined by
 \begin{align}\label{Vdefa}\Vcal(\phi)=\vol(\Ccal\cap\phi^{-1}(0,1])=\int_{\Ccal\cap\phi^{-1}(0,1]}\dvol\end{align}
  Let 
$\pi:\Ccal_{\phi}\to \Omega=S^n\cap\Ccal$ be the radial projection $\pi(x)=x/\|x\|$. 
Let $\operatorname{dB}$ be the induced  volume form on $S^n$. We  compute this integral in
polar coordinates, so
 $\dvol=r^{n} \operatorname{dr}\wedge\pi^*\operatorname{dB}$.

Given $\phi\in\Ccal^*$, 
 there is $v\in V$ such that $\phi(x)=\langle x,v\rangle$ for all $x\in V$. Let $y=\pi(x)$. If $\phi(x)=1$,  then  \begin{align}\label{reqtn}
r=r(x)=\|x\|=\|x\|/\|\pi x\|=\phi(x)/\phi(\pi x)=\langle y,v\rangle^{-1}\quad{\rm and}\qquad x=\langle y,v\rangle^{-1}y\end{align}
Using polar coordinates
\begin{align}\label{Vformula}\Vcal(\phi)=\int_{C_{\phi}\cap\phi^{-1}(0,1]}r^n\operatorname{dr}\wedge\pi^*\operatorname{dB}=
\int_{\Omega}\left(\int_0^{\langle y,v\rangle^{-1}} r^n \operatorname{dr}\right)\operatorname{dB}_y
=(n+1)^{-1}\int_{\Omega}\langle y,v\rangle^{-n-1}\; \operatorname{dB}_y\end{align}
For $q\in\Ccal$, the set $\Ccal^*_q=\{\phi\in\Ccal^*:\ \phi(q)=1\ \}$ 
is the intersection of a hyperplane in $V^*$ with $\Ccal^*$, and has compact closure. Elements of $\Ccal^*_q$ correspond to hyperplanes 
in $V_0$ that contain $q$ and separate $\Ccal$.

\begin{proposition}\label{Vprop}The volume function has the following properties:
\begin{itemize}
\item[(i)]  $\Vcal$ is smooth and strictly convex.\\
\vspace{-0.2cm}\item[(ii)] $\Vcal(\phi)\to\infty$ as $\phi\to\Fr\Ccal^*$.\\
\vspace{-0.2cm}\item[(iii)] $\Vcal(t\cdot\phi)=t^{-n-1}\Vcal(\phi)$ for all $t>0$.\\
\vspace{-0.2cm}\item[(iv)]
There is a unique $\phi\in \Ccal^*_q$ at which $\Vcal|\Ccal^*_q$ attains a minimum and   $q=\mu(\Ccal_{\phi})$.\end{itemize}
\end{proposition}
\begin{proof} We use the inner product to identify $V$ with $V^*$
so that $\Ccal^*$ is a subset of $V$ and \[\Ccal_q^*=\{x\in \Ccal^*:\langle x,q\rangle=1\}\]
 We also
 regard $\Vcal(\phi)$ as the function $\Wcal(v)$  given by the right hand side of \Cref{Vformula} and prove  corresponding statements
for $\Wcal$. 

For fixed $y$, the integrand $f(v)=\langle y,v\rangle^{-n-1}$ in \Cref{Vformula} is a smooth and convex function of $v$.
It follows that $\Wcal(v)$ is a smooth, {\em strictly} convex function of  $v$. This proves (i).

If $0\ne\psi\in\Fr\Ccal^*$, then there is $0\ne x\in\Fr\Ccal$ with $\psi(x)=0$.
Thus $\RR^+\cdot x\subset\cl\Ccal_{\psi}$, so $\Ccal_{\psi}$ is not compact. Moreover $\Ccal_{\psi}$ is convex
and has non-empty interior, so $\vol(\Ccal_{\psi})=\infty$. It easily follows that $\Vcal(\phi)\to\infty$
as $\phi\to\psi$. This proves (ii), and (iii) follows from \Cref{Vformula}. It follows from convexity and (ii)
that $\Vcal|\Ccal^*_q$ has a unique critical
point, and it is a minimum.

   The gradient of $\Wcal(v)$ is
\begin{align}\label{gradV} \nabla \Wcal =-\int_{\Omega} \langle y,v\rangle^{-n-2} y\;\operatorname{dB}_y
\end{align}
 From the definition of $\Ccal^*_q$ it follows that the condition for a critical point  is that $\nabla \Wcal\in\RR\cdot q$.
 We use radial projection  $\pi$ to change variables in \Cref{gradV}
to obtain  \Cref{newgradV}.

\begin{figure}[h]
                       \centering 	 \includegraphics[scale=0.7]{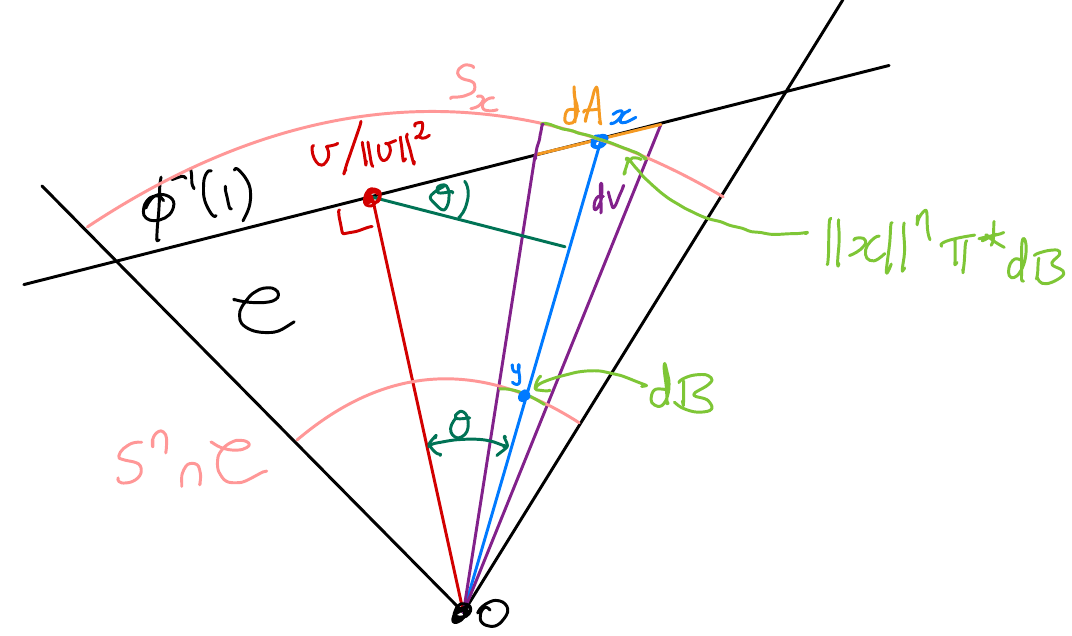}
 \caption{Volume forms on $\Omega\subset S^n$ and $\Ccal_{\phi}$} \label{dAdBfig}
\end{figure}
  Refer to \Cref{dAdBfig}. Let $v\in V$ be dual to $\phi$, so $\phi(x)=\langle x,v\rangle$. Thus $\phi(v\|v\|^{-2})=1$ and
$v\|v\|^{-2}$ is the point on $\phi^{-1}(1)$ closest to $0$. Thus
the distance of $\phi^{-1}(1)$ from $0$ is $\|v\|^{-1}$. Given $x\in \Ccal_{\phi}$,  then $1=\phi(x)=\langle x,v\rangle$. Set $y=\pi(x)$,  then $\|y\|=1$,
and define $\cos\theta=\langle y,v\rangle/\|v\|$. Since $\langle x,v\rangle=1$ it follows that  $x=y/\langle y,v\rangle$ and
 so $\|x\|=1/\langle y,v\rangle$.

  Let $S_x\subset\RR^{n+1}$ denote the sphere with center $0$ and radius $\|x\|$. The volume form on $S_x$ is $d_{S_x}=\|x\|^{n}\pi_1^*dB_y$, where 
  $\pi_1:S_x\rightarrow S^n$ is radial projection. 
 Let 
 $d A_x$ be  the  volume element on $\Ccal_{\phi}$. Then $dA_x=(\cos\theta)^{-1}\pi_2^*d_{S_x}$,
 where $\pi_2:\Ccal_{\phi}\rightarrow S_x$ is radial projection.  Then 
  $\pi=\pi_1\circ\pi_2:\Ccal_{\phi}\rightarrow S^n$ is radial projection, and
  combining gives
\begin{align}\label{dAdB}
d A_x=(\cos\theta)^{-1}\|x\|^{n}\; \pi^*d B_y=  \|v\|\langle y,v\rangle^{-n-1}\; \pi^*d B_y
\end{align}
It follows from \Cref{gradV,dAdB} that
\begin{align}\label{newgradV}\nabla \Wcal =-\|v\|^{-1}\int_{\Ccal_{\phi}}x \; \operatorname{dA}_x=-\|v\|^{-1}\mu(\Ccal_{\phi})\int_{\Ccal_{\phi}}  \operatorname{dA}_x \end{align}
Since the condition for a critical point  is $\nabla \Wcal\in\RR\cdot q$, this happens when
$$\mu(\Ccal_{\phi})\in\RR\cdot q$$
The centroid $\mu(\Ccal_{\phi})$ is in $\Ccal_{\phi}$, and $q\in \Ccal_{\phi}$. Hence at the minimum we have $\mu(\Ccal_{\phi})=q$, proving (iv).
\end{proof}

\begin{proof}[Proof of  \Cref{SPdef}]   Clearly $\Dcal$ is a closed convex subset of $V$. 
Given $q\in \Ccal$, there is $\phi\in\Ccal^*$  given by (\ref{Vprop})(iv). Then by (\ref{Vprop})(iii) there is $t>0$ such that $\Vcal(t\cdot\phi)=1$. Replace $\phi$ by $t\cdot\phi$ and replace $q$ by $t^{-1}q$. Then
$\Vcal(\phi)=1$ and the hyperplane $T_{\phi}=\phi^{-1}(1)$ contains $q$. The halfspace $H={\phi}^{-1}[1,\infty)$ has the property that
$\vol(\Ccal\setminus H)=1$, so $\Dcal\subset H$.

Suppose that $\Vcal(\psi)=1$
and $\psi\ne \phi$,  then $\lambda=\psi(q)>1$
since otherwise if $\psi'=\lambda^{-1}\psi\in\Ccal^*_q$ and $\Vcal(\psi')=\lambda^{n+1}\le 1$. This contradicts that $\phi$ is the unique minimum of $\Vcal|\Ccal^*$.
Hence ${\psi}(q)>1$ so $q\in\bdy\Dcal$ and $T_{\phi}$ is a supporting hyperplane to $\Dcal$ at $q$. Also $q=T_{\phi}\cap\Dcal$ is the centroid of $\Ccal_{\phi}$ by (iv). 
This proves (2) and (5).
It also
follows that $\Dcal\cap(\RR^+\cdot q)=[1,\infty)\cdot q$, which proves (3).

Suppose that $T$ is any supporting hyperplane at $q$. Then there is $\psi\in V^*$ with $T=\psi^{-1}(1)$ and $\psi(\Dcal)\ge 1$. 
The same reasoning shows that $\phi=\psi$, thus $T_{\phi}$ is the unique supporting hyperplane at $q$, which proves (6).
Now (4) follows from the fact that $\SL(\Ccal)$ preserves volume. 

It only remains to prove (1). Refer to \Cref{smallspace}. 
Since $\Dcal$ is convex, it suffices to prove $\Dcal$ contains no point in the frontier of 
$\Ccal$.
\begin{figure}[h]
                       \centering 	 \includegraphics[scale=0.7]{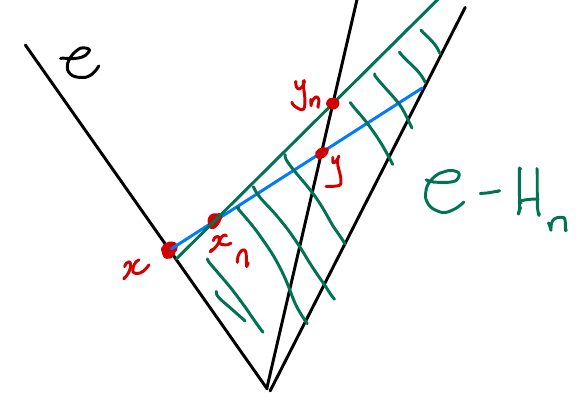}
 \caption{} \label{smallspace}
\end{figure}
Suppose $x\in\Dcal\cap\Fr\Ccal$. Then $x\ne 0$ by definition of $\Dcal$.  Choose $y\in\Ccal$ and let $U\subset V$ be the two dimensional vector subspace that contains both $x$ and $y$.
Let $x_n$ be a sequence in $(x,y)$ that converges to $x$. When $x_n$ is close to $x$, there is a halfspace $H'_n$ with $x_n\in\bdy H'_n$ and
$\vol(\Ccal\setminus H'_n)$ is very small. Thus if $\phi_n\in\Ccal^*_{x_n}$ is as given by (iv),  then
$\epsilon_n=\Vcal(\phi_n)\to0$ as $n\to\infty$. Since $x$ is above $\bdy H_n=\phi_n^{-1}(1)$ it follows that $y$ is below $y_n=(\RR\cdot y)\cap \bdy H_n$.

 Let $\psi_n=\epsilon_n^{1/(n+1)}\phi_n$,  then by (\ref{Vprop})(iii), we have $\Vcal(\psi_n)=1$. Then $\Dcal\subset\psi_n^{-1}[1,\infty)$.
The hyperplane $T_n=\psi_n^{-1}(1)$
intersects $\RR\cdot y$ at $z_n=\epsilon_n^{-1/(n+1)}y_n$. Since $\|y_n\|>\|y\|$ it follows that $z_n\to\infty$ as $n\to\infty$. Let $r_n$ be the distance of $z_n$ from $\bdy\Ccal$. Then $r_n\to\infty$. Half
the ball of radius $r_n$ center at $z_n$ is below $T_n$ and in $\Ccal$.  But this volume is at most $1$. This contradicts the existence of $x$, proving (1).
\end{proof}

\begin{corollary}\label{Thetaprop} If $\Omega$ is properly convex, then $\Theta:\Ccal\Omega^*\rightarrow\Ccal\Omega$ is a homeomorphism that maps  rays to rays, and $[\Theta]:\Omega^*\rightarrow\Omega$ is a homeomorphism.\end{corollary}
\begin{proof} 
Given $p\in\Omega$ there is a unique $x\in\bdy\Dcal$ 
with $p=[x]$. By (\ref{Vprop})(iv) there is a unique $\phi\in\Ccal^*$ with $\phi(x)=1$
and $\Vcal(\phi)=1$. Moroever $\mu(\Ccal_{\phi})=x$ and so $[\Theta][\phi]=p$.
Hence $[\Theta]$ is a bijection.
 Clearly $\Theta$ is continuous. By (\ref{Vprop})(ii)   $[\Theta]$ is proper, and thus a homeomorphism.
  Since $\Theta(t x)=t^{-n-1}\Theta(x)$, rays are mapped to rays. It follows that $\Theta$ is also a homeomorphism.
 \end{proof}
The dual action of $A\in\SL _{\pm}V$ on $\PP V^*$
 is given by $A[\phi]=[\phi\circ A^{-1}]$. If $\Gamma\subset\SL_{\pm} V$ and preserves $\Omega$, then the dual action 
 of $\Gamma$ preserves $\Omega^*$. It is clear that $\Theta$ is equivariant with respect to these actions.
 It directly follows from (\ref{C1domain}) and (\ref{Thetaprop}) that:

  \begin{corollary}[Vinberg]\label{dualmfd} If $M=\Omega/\Gamma$ is a closed, properly convex manifold, then $M^*=\Omega^*/\Gamma$
 is a properly convex manifold that is homeomorphic to $M$. Thus $\pi_1M\cong\pi_1M^*$. 
  \end{corollary}

If $M=\Omega/\Gamma$, then $M$ is called {\em $C^1$} if  $\Omega$ is $C^1$. It follows from
 (\ref{C1domain}) that:
\begin{corollary}\label{C1}  $M$ is strictly convex if and only if $M^*$ is $C^1$.\end{corollary}

  \begin{corollary}\label{dualcpct} If $M=\Omega/\Gamma$ is a closed, strictly convex manifold, then $M^*=\Omega^*/\Gamma$
 is a closed, strictly convex manifold. We have $\pi_1M\cong\pi_1M^*$ and call  $M^*$ the {\em dual } of $M$.
 Moreover, $M$ is $C^1$.
  \end{corollary}
   \begin{proof}   Since $M$ is strictly convex, it is properly convex, so $M^*$ is a properly convex manifold,
   and $\pi_1M\cong\pi_1M^*$. By (\ref{htpyequiv}) and  (\ref{stricthomeo}) $M^*$ is closed, and 
 strictly convex. Thus $M=(M^*)^*$ is $C^1$ by (\ref{C1domain}).
    \end{proof}

The centroid of a bounded open convex set $\Omega\subset {\mathbb R}^n$ is a distinguished point in $\Omega$.
We wish to define something similar for  certain subsets of the sphere $S^n=\{x\in\RR^{n+1}:\|x\|=1\}$. Imagine a room that contains a transparent globe close to one wall, with a light source at the center of the globe. The shadow of Belgium appears on the wall.
You can  rotate the globe so that the centroid of this shadow is  the point $p$ on the wall closest to the center of the globe. 
The point on the globe that projects to $p$ is called the
{\em (spherical) center} of Belgium.

The open hemisphere that is
the $\pi/2$ neighborhood of $y\in S^n$ is $U_y=\{x\in S^n: \langle x,y\rangle>0\}$.
Radial projection $\pi_y:U_y\longrightarrow {\operatorname{T}_y}S^n$ 
from the origin onto the tangent space
to $\SS^n$ at $y$ is given by $$\pi_y(x)=\frac{x-\langle x,y\rangle y}{\langle x,y\rangle}$$
where we identify $S^n$ with $\SS^n$.

\begin{definition} If $\Omega\subset \SS^n$ is properly convex and $\Omega^*=\{y\in S^n:\ \langle x,y\rangle>0\ {\rm \ for \ all \ } x\in\cl(\Omega)\}$, then $y\in \Omega^*$ is called
a {\em (spherical) center of $\Omega$}  if 
 $\cm(\pi_y(\Omega))=\pi_y(y)$.
 \end{definition}

If $A\in\O(n+1),$ then $A$ is an isometry of the inner product and so $(A\Omega)^*=A(\Omega^*)$, and  $Ay$ is a center of $A\Omega$
if  $y$ is a center of $\Omega$. The following property of $\bdy\Dcal$ is important for the proof of the Closed Theorem.
\begin{theorem}\label{centerexists} If $\Omega\subset S^n$ is properly convex, then  it has a unique spherical center $[x]$. 
Moreover
 $x$ is the point on $\bdy\Dcal$ that minimizes $\|x\|$.
\end{theorem}
\begin{proof} Let $\phi\in V^*$ be given by $\phi(y)=\langle y,x\rangle$.
The tangent plane to $\bdy\Dcal$ at $x$ is $\phi^{-1}(1)$ and is orthogonal to $x$. 
Thus $\pi_x(\Omega)=\Ccal_{\phi}$ and $\mu(\Ccal_{\phi})=x$.
\end{proof}
 
 \begin{corollary} If $\Omega\subset\RP^n$ is properly convex and $p\in \Omega$, then
 there is an affine patch $\RR^n\subset\RP^n$ such that $p$ is the centroid of $\Omega$ in $\RR^n$.
\end{corollary}
\begin{proof} There is a unique $x\in\bdy\Dcal$ with $p=[x]$.
Choose an inner product on $V$ so that $x$ is the closest point to $0$ on $\bdy\Dcal$.
The required affine patch is $\RP^n\setminus\PP(x^{\perp})$.
\end{proof}

 \section{Open}\label{opensec}

If $M$ is a closed $(n-1)$-manifold,  then $\pi_1M$  is finitely generated by some set of elements $g_1,\cdots,g_k$. Now $\rho\in \Rep(M)$
is uniquely determined by the $k$ elements $\rho(g_i)\in\PGL(n,\RR)$. 
The Veronese embedding embeds real projective space  (and thus $\PGL(n,\RR)$) into a 
 Euclidean space.
This collection of elements can then be thought of as one point in a finite dimensional 
real vector space $V$.
This gives an injective map of $\Rep(M)$ into $V$. The {\em Euclidean topology} on $\Rep(M)$
is the subspace topology (of the standard topology on $V$)  on the image of this map.  In this topology
$\rho$ and $\rho'$ are close if there are matrix representatives of  $\rho(g_i)$ and $\rho'(g_i)$ that are close.

\begin{theorem}[Ehresmann-Thurston principle \cite{MR2110758}]\label{ET} Suppose that $M$ is a closed manifold of dimension $n$. Then the subset of
$\Rep(M)$ consisting of the holonomies of real projective structures on $M$ is open in the Euclidean topology.\end{theorem}
\begin{proof}  Let $\widetilde{M}$ be the universal cover of $M$.
We identify $\pi_1M$ with the group of covering transformations of $\widetilde M$.
Choose a triangulation $\Tcal_M$ of $M$. 
Let $\Tcal$ be the lifted triangulation on $\widetilde M$. 
 
Suppose that $\rho\in\Rep(M)$
and  there is a collection of maps
$$f_{\sigma}:\sigma\rightarrow\RP^n$$ one 
for each simplex $\sigma$ in $\Tcal$, satisfying the following properties.
\begin{enumerate}
\item {\em Equivariance.}
If $\sigma'=g\cdot\sigma$ for some $g\in\pi_1M$,  then 
$$f_{\sigma'}=\rho(g)\circ f_{\sigma}$$
\item {\em Compatibility.} If $\tau$ is a face of $\sigma$,  then
$$f_{\tau}=f_{\sigma}|\tau$$
\end{enumerate}
Compatibility implies there is a  map $F:\widetilde{M}\longrightarrow \RP^n$ such that $F|\sigma=f_{\sigma}$. Equivariance implies
$$F(g\cdot x)=\rho(g)(x)$$
Suppose that  $\dev_0:\widetilde{M}\rightarrow\RP^n$ is the developing map, 
and $\rho_0\in\Rep(M)$ is the holonomy of some real projective structure on $M$. If $\rho=\rho_0$, 
 then the collection $f_{\sigma}=\dev_0|\sigma$ satisfies the required properties and $F=\dev_0$.

Choose the triangulation
$\Tcal_M$ of $M$ to be by projective simplices.
We  show below that if $\rho$ is close enough to $\rho_0$,  then the maps $f_{\sigma}$ may be chosen to be close to $\dev_0|\sigma$
on any finite set of simplices. If this finite collection contains a neighborhood, $U$, of a fundamental domain $D$, this will ensure $F|U$ is a local homeomorphism.
By equivariance $F$ is a local homeomorphism everywhere. Then $F$ is a developing map for $\rho$, and this defines a projective structure on $M$
with holonomy $\rho$,

It suffices to define $f_{\sigma}$ on one simplex in each $\pi_1M$ orbit. Then the equivariance  property defines the maps on the rest of the orbit. In order to
ensure compatibility, we define the maps in order of increasing dimension of simplices. At each stage we have a well defined equivariant map $F_k:\Tcal^k\rightarrow\RP^n$
 defined on
the $k$-skeleton $\Tcal^k$ of $\Tcal$. 
The extension process is to choose one $(k+1)$-simplex $\sigma$ in each $\pi_1M$ orbit, and define $f_{\sigma}$ to be a homeomorphism
onto a projective simplex that has boundary $F_k(\bdy\sigma)$. By choosing $\rho$ close enough to $\rho_0$
we can arrange that the projective simplex $f_{\sigma}(\sigma)$ is   non-degenerate and  close to $\dev_0(\sigma)$.
The induction starts by choosing one vertex $v$ in each $\pi_1M$ orbit
and defining $f_v(v)=\dev_0(v)$. 
   \end{proof}

The next goal is to characterize properly convex manifolds in such a way that the above proof applies.
A set is {\em locally convex} if every point
has a convex neighborhood. There is a basic local-to-global principle for convexity. 
\begin{proposition}\label{convexhyper} If $K$ is a closed, connected, locally convex subset of $\RR^n$ that is the closure of its interior, then $K$ is convex.
\end{proposition}
\begin{proof} Suppose that  $[x,y]$ is a segment in  $K$. Then $[x,y]\cap\bdy K$ is closed.  If $(x,y)$ contains a point in $\bdy K$, then  by local
convexity $[x,y]\cap\bdy K$ is open in $[x,y]$. Thus $[x,y]\subset\bdy K$.
 It follows that  the union of the segments in $K$ with one endpoint $x$ is  open in $K$. Since $K$ is closed, and the limit of segments is a segment, this subset is also closed. Since $K$ is connected this set equals $K$.\end{proof}

 Local convexity only needs to be checked at points in $\bdy K$. Suppose that $K$ is contained in the upper halfspace $x_1\ge0$ in $\RR^{n+1}$. Then $K\cap(0\times\RR^{n})$ 
is a convex subset of $\RR^n$ if the subset $S\subset \bdy K$ 
where $x_1>0$ is a locally convex hypersurface in $\RR^{n+1}$. Informally: a mountain with a convex surface has a convex base.
We now extend this to show that a manifold is properly convex.

A projective manifold is {\em convex} if every pair of points is contained in a segment.
If $M$ is a closed projective manifold, the argument in (\ref{convexhyper}), that the union of the segments starting at a point is closed, fails. For example, consider the projective 
torus $T$ that is the quotient of $\RR^2_0$ by a cyclic group generated by a homothety. Then $T$ is not convex. Let $\pi:\RR_0^2\rightarrow T$ be the projection.
 If $x\in\RR^2_0$ and $y_n\in\RR^2_0$ is a sequence that converge to $-x$
then the Hausdorff limit of the segments $\pi[x,y_n]$ in $T$ is not connected. It is the pair of rays, $\pi[x,0)\cup \pi(0,-x]$. Each ray maps into in  one of
the pair  of circles
$\pi(\RR_0\cdot x)\subset T$.

\halfgap A smooth hypersurface $S\subset\RR^n$ is {\em Hessian-convex} if 
the surface is locally the zero-set of a smooth, real-valued function with positive-definite Hessian. Suppose $\Omega\subset\RPn_+$ is properly convex and $\Ccal\Omega=\{v\in \RR^{n+1}_0:[v]\in\Omega\}$ is the corresponding convex cone. 
Suppose $M=\Omega/\Gamma$ is a compact, and properly convex, $n$--manifold.
Then $\widetilde W=\Ccal\Omega/\Gamma\cong M\times (0,\infty)$ 
is a properly convex affine $(n+1)$--manifold. We may quotient out by a homothety to obtain a compact affine $(n+1)$--manifold $W\cong M\times S^1$.

By (\ref{SPdef}) there is a  hypersurface $\bdy\Dcal\subset \Ccal\Omega$ that is $\Gamma$--invariant and Hessian-convex away from $0$.
Then $Q=\bdy \Dcal/\Gamma$ is a compact, Hessian-convex, codimension--1 submanifold of $W$. 
Let $K\subset\RP^{n+1}$ be
the closure of $\Dcal$ and regard $\RP^{n+1}=\RR^{n+1}\sqcup\RP^n_{\infty}$. The interior of $K\cap\RP^n_{\infty}$ can be identified with $\Omega$.
The existence of the hypersurface $Q$ in $W$ {\em implies} the existence of $\Dcal\subset\RR^{n+1}$, which {\em implies} $\Omega$ is properly convex,  by the reasoning above.

By the Ehresmann-Thurston principle, deforming $\Gamma$ a small amount to $\Gamma'$
 gives a new projective $n$--manifold $M'\cong M$, and a new affine $(n+1)$--manifold $W'\cong M\times S^1$. 
 Now $W'$ contains  a hypersurface $Q'$ that is Hessian-convex, provided
the deformation of the developing map is small enough in $C^2$. It then follows
that  $M'$ is properly convex. This is the approach taken in \cite{CLT2} for non-compact manifolds.

Here, we choose to
 work with a piecewise linear submanifold $Q$ in place of a smooth one.  Hessian convexity 
 is replaced by the condition
 that at each vertex of the triangulation the determinants, of certain matrices formed by the relative positions of vertices, are strictly positive.
This ensures local convexity near the vertex. This version of convexity is preserved by small $C^0$--deformations
of the developing map. Thus it is easier to apply.
We start by reviewing these ideas for hyperbolic manifolds.

\medskip

The hyperboloid model of the hyperbolic plane  is the action of $\O(2,1)$ on
  the surface $S\subset\RR^3$  given by $x^2+y^2-z^2=-1$. If we identify $\RR^3$ with $\RP^3\setminus\RP^2_{\infty}$,
 then we can regard this as an affine action on 
  $\RP^3$ by
using the affine subgroup $\O(2,1)\oplus(1)\subset\SL(4,\RR)$.
The open disc  $$D=\{[x:y:z:0]:\ x^2+y^2<z^2\}\subset\RP^2_{\infty}$$ has  the same frontier as $S$,
   and $(\O(2,1),D)$ is the projective (Klein) model  of the hyperbolic plane. See \Cref{SandOmega}.
 
\begin{figure}
                       \centering\includegraphics[scale=0.7]{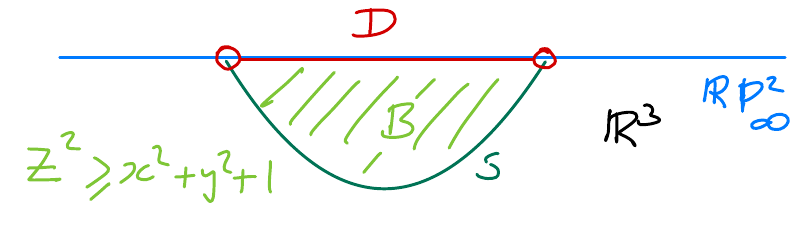} 
 \caption{The Hyperboloid and Klein models of $\HH^2$} \label{SandOmega}
\end{figure}

 The surface $\cl(S\cup D)\subset\RP^3$ bounds a closed $3$-ball $B\subset\RP^3$. 
 Let  $\Omega=B\setminus\Fr D\cong D\times I$.
 The fact that $S$ is
 a strictly convex surface {\em implies} $\cl(S\cup D)$
 is a convex surface in some affine patch which {\em implies} $\Omega$ is properly convex, and this {\em implies} $D$ is properly convex.
 If one thinks of $B$ as an upside down mountain, then the convexity of the surface $S$ implies the convexity of $D$.

  The action of $\O(2,1)\oplus(1)$ preserves $\Omega$.
    Suppose $\Gamma\subset\O(2,1)$ and $\Sigma=D/\Gamma$ is a  hyperbolic surface. Let $\Gamma'=\Gamma\oplus(1)$,  then $N=\Omega/\Gamma'\cong\Sigma\times[0,1]$
is a properly convex 3--manifold with one flat boundary component $\Sigma=D/\Gamma$ and one
strictly convex boundary component $M=S/\Gamma$. The fact that $M$ is a strictly convex surface {\em implies} $\Sigma$ is properly convex.
We will generalize this construction to arbitrary properly convex manifolds in place of $\Sigma$. But first we 
divide out by a homothety.

Consider the cone $\Ccal=\{\lambda\cdot x:x\in S,\ \ \lambda>0\}$. There is a product structure $\widetilde \phi:S\times\R^+ \to\Ccal$ given by $\widetilde\phi(x,\lambda)=\lambda\cdot x$
 on $\Ccal$ that is preserved by $\Gamma'$. Let $H\subset\GL(4,\RR)$ be the cyclic
group generated by $\Diag(2,2,2,1)$. Then $\Gamma'$ centralizes $H$ and the group $\Gamma^+\subset\GL(4,\RR)$ generated by $\Gamma'$ and $H$
preserves $\Ccal$ and $W:=\Ccal/\Gamma^+$ is a closed $3$--manifold homeomorphic to 
$\Sigma\times S^1$.

In what follows $\RR^+=\{x\in\RR:x>0\}$, and $S^1=\RR^+/\exp(\ZZ)$, has universal cover $\RR^+$.
There is a product structure $\phi:\Sigma\times S^1\to W$
covered by $\widetilde \phi$, and the surfaces $\phi(\Sigma,\theta)$ are convex.
The reader might contemplate all this in the case of one dimension lower, where $\Gamma\subset\SO(1,1)$
is generated by a hyperbolic and $\Ccal/\Gamma^+$ is an affine structure on $S^1\times S^1$. 

\medskip
We now return to the general setting.  
 The {\em radiant-affine group} is  the subgroup of the affine group acting on $\RR^n$ that fixes the origin,
$$\FG(n)=\bpmat\GL(n,\RR) & 0\\ 0 &1\epmat\subset \Aff(\RR^n)$$  {\em Radiant-affine geometry}  is 
$(\FG(n),\RR^n_0)$ and is a subgeometry of affine geometry,  so we may use affine notions. It is  also isomorphic to a subgeometry  of projective geometry under the embedding \[\RR^n_0\hookrightarrow\{[x:1]:   x\in\RR^n_0\}\subset\RP^n\]
The subgroup $t \Id_n\oplus 1$ with $t>0$ is called the {\em homothety flow}. 
This subgroup acts  on  $\RR^n_0$ by homotheties:  $x\mapsto t\cdot x$. 
The  quotient  of the tautological line bundle  by a cyclic group of homotheties
 is the  {\em tautological circle bundle} over $M$. 
\begin{definition}\label{RAbundle} If $M$ is a closed $(n-1)$--manifold, then a {\em tautological circle bundle} structure on $M\times S^1$
is a radiant-affine geometry structure $(\dev,\rho)$ such that:
\begin{enumerate}
\item the generator of $\pi_1(*\times S^1)$ acts on $\RR^n_0$ by 
$x\mapsto 2 x$
\item $\det(\rho\pi_1(M\times 1))\subset\{\pm1\}$
\end{enumerate}
\end{definition}

 Let $\pi:\RR^n_0\rightarrow S^{n-1}$ be radial projection $\pi(x)=x/\|x\|$.  
The second condition implies that \[(\pi\circ\dev|(\widetilde{M}\times 1),\rho|\pi_1(M\times 1))\] is a positive
projective structure on $M$. Conversely a positive projective structure on $M$ determines a
radiant-affine geometry structure on $M\times S^1$.  Multiplication by $S^1$ on the right factor of $W=M\times S^1$ corresponds by the developing map to homotheties. We call  all of these actions {\em homothety} or the {\em radial flow}.

Suppose $W$ is an affine $n$--manifold, and $S$ is a hypersurface in $W$. Then
$S$
 is a {\em convex hypersurface},  and it is {\em locally convex}, if for every point $p\in S$
 there is an $n$--dimensional submanifold $Q\subset W$ with $Q\cong\bdy Q\times[0,1)$ and $\bdy Q\subset S$, and $\bdy Q$ is a neighborhood
 of $p$ in $S$ and $\dev(U)\subset\RR^n$ is convex.
If $W=M\times S^1$ is a tautological circle bundle, then such $S$ is   {\em outwards} convex
if, in addition, $t\cdot \bdy Q\cap Q\ne\emptyset$ for some $t>1$. In other words $\dev\widetilde S$ is locally-convex {\em away} from $0$ in $\RR^{n+1}$.

A locally convex hypersurface $S$ is {\em strictly convex} if every connected flat subset  of $S$ is a single point.
If, instead, every connected flat subset of $\widetilde S$ is contained in a compact set,
then $S$ is called {\em strongly-convex}.  We will make use  of strongly convex when $S$ has a triangulation by  flat simplices,
no two of which are in the same hyperplane. In this case $S$ is called {\em simplicial convex}.

\halfgap

\begin{definition} A tautological circle bundle structure on $M\times S^1$ is {\em radial convex}
if $M\times\theta$ is an  outwards strongly-convex hypersurface in $M\times S^1$ for some (and hence all) $\theta\in S^1$.
\end{definition}

 \begin{theorem}\label{sufficient} If $M\times S^1$ is compact and radial convex, then the associated projective structure on $M$ is properly convex.
 \end{theorem}
 \begin{proof} Let $n=\dim M$. In this proof we consider the image of several sets under the developing map and  use the  notation $X=\dev(X')$.
 
Let $N=M\times S^1$ and $\pi_{_{N'}}:\widetilde N'=\widetilde M\times\RR^+\rightarrow N$ be the universal cover,
and $\dev:\widetilde N'\rightarrow\RR^{n+1}_0$ the developing map for $N$.
Let $R'=\widetilde M\times t$ for some $t>0$, and choose a basepoint $p'\in R'$.
Let $V\subset\RR^{n+1}$ be a $2$-dimensional vector subspace that contains $p=\dev(p')$. 
Since $V$ is preserved by homothety, $\dev^{-1}V = X\times \RR^+$ for some $1$-submanifold
$X\subset \widetilde M$.
Let $C'$
be the component of $R'\cap \dev^{-1}V$ that contains $p'$. Then $C'$ is a 
connected curve in $\dev^{-1}V$ without endpoints
  that is transverse to
the homothety flow and convex outwards. 

The curve $C=\dev(C')$ is immersed in $V_0$, contains $p$,  is everywhere transverse to the radial direction, and convex outwards: radially away from $0$.

\halfgap

{\bf Claim 1:}  $\dev|C'$ is injective. 

Let $\pi:V_0\to S^1$ be radial projection $\pi(x)=x/\|x\|$. 
We  show that $\theta=\pi\circ \dev:C'\to S^1$ is injective. Since $C'$ is transverse to the radial direction in $\widetilde N'$, it follows that
 $\theta$ is an immersion. Let $\ell\subset V$ be the tangent line to $C$ at $p$. Suppose
 $q\in C\cap\ell$ is distinct from $p$. Then at some point $r$ in $C$ between $p$ and $q$, the distance 
 of $r$ from $\ell$ is a maximum. This contradicts that $C$ is convex outwards at $r$.
  Hence $\pi(\ell)$ is an open semi-circle in $S^1$ that contains $\pi (C)$. Thus $\theta$ is an immersion of $C'$ into this arc,  therefore $\theta$ is injective, and it follows that $\pi$ is injective. This proves Claim~1.

\halfgap
 
{\bf Claim 2:} $C$ is a closed subset of $V$. 

Otherwise $C$ limits on a point $w$ in $V$.  Let $\widetilde\kappa:[0,1)\rightarrow V$ be an affine map
with image
$[p,w)$.
\begin{figure}[h]
                       \centering\includegraphics[scale=0.7]{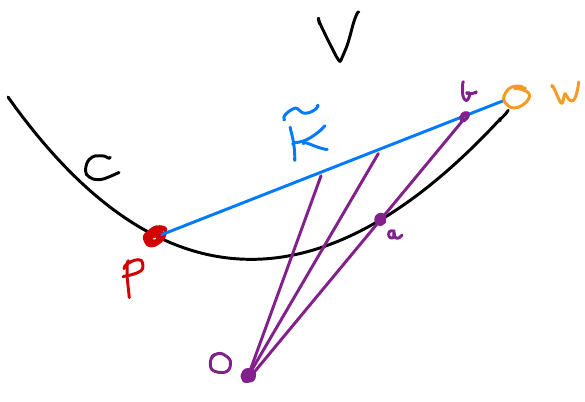}
 \caption{$\widetilde\kappa$ in $V$} \label{kappa}
\end{figure}
Refer to \Cref{kappa}. Now $C'\times\RR^+\subset\widetilde N'$
so there is an affine ray $\widetilde\kappa':[0,1)\rightarrow C'\times\RR^+$ with $\kappa=\dev\circ\widetilde\kappa'$
that starts at $\widetilde\kappa'(0)=p'$. This ray  leaves every compact set in $\widetilde N'$, since otherwise it can be extended, and this would extend $C'$.
Thus $\kappa'=\pi_{_{N'}}\circ\widetilde\kappa':[0,1)\rightarrow N$ 
 spirals in towards, and accumulates on, 
  some subset $F$ of $M\times t\subset N$. In other words $F$ is the forward limit set of $\kappa'$. 
  To justify this, in \Cref{kappa} the ray $\widetilde\kappa$ approaches $C$ in $V$. However in $N$ the ray $\kappa'$ has infinite length. Choose
  a product metric on $N=M\times S^1$. The universal cover of $S^1$ is $\RR^+$ and we use the metric on $\RR^+$ given by $d(a,b)=|\log(a/b)|$.
  From \Cref{kappa} one can see the distance in $N$ of points on $\kappa'$ from $M\times t$ in the $S^1$ direction goes to $0$.

 Now $F$ is flat because $\kappa'$
 is affine. The closure of some component of $\pi_{_{N'}}^{-1}(F)$ in $\widetilde N'$ is not compact, otherwise $\kappa'$
converges to a limit point in $F$ that maps to $w$. This
  contradicts that $M\times t$ is strongly locally convex, and  thus contradicts that $M\times S^1$ is radial convex. 
 This proves Claim 2.
 
\halfgap

Thus every point in $R'$ is connected to $p'$ by some segment in $\widetilde N'$. Using the radial flow it follows that $\widetilde N'$
is starshaped from $p'$. But $p'$ was arbitrary, so $\widetilde N'$ is convex. Hence $\dev:\widetilde N'\rightarrow \RR^{n+1}_0$ is injective.
Moreover since $\kappa$ lies above $C$ it follows that  $\
X=\dev(\widetilde M\times [1,\infty))$ is convex. 
Now $\bdy X=\dev(\widetilde M\times 1)$ is a properly embedded convex hypersurface in $\RR^{n+1}$. Thus $X$ is properly convex in $\RP^{n+1}$.
Let $Y$ be the closure of $X$ in $\RP^{n+1}$. Then $Y$ is properly convex and
 it follows that $\RP^n_{\infty}\cap Y$ is properly convex. Now $Y$ is the image of the projective structure on $M$ under the developing map
 into $\RP^n_{\infty}$. Hence $M$ is properly convex. \end{proof}

If $W$ is affine,  then a  hypersurface, $S\subset W$, is  {\em simplicial}  if it has a triangulation
by flat simplices.   
If $S$ is locally convex and simplicial, and every connected flat subset of $S$ is contained in a simplex in $S$,
then $S$ is called {\em simplicial convex}. 
These are a substitute for {\em strictly convex}. Observe that a locally convex simplicial surface is simplicial convex if and only if adjacent
simplices the same dimension as $S$ are never in the same hyperplane.

\begin{lemma}\label{stabilitylemma} Suppose that $S$ is a simplicial convex hypersurface in an affine $n$-manifold $W$. Then
 for each vertex $v$ of $S$, there is a neighborhood $U(v)$,
such that the hypersurface $S'$ obtained  
 by moving each vertex $v$ inside $U(v)$ is simplicial convex. \end{lemma}
 \begin{proof}  This is a local question, so we assume $W=\RR^n$. The general case follows using the developing map.
   The simplicial convex condition at the vertex $v$ is equivalent to the positivity of the determinants
 of certain $n\times n$ matrices formed by vectors of the form $u-v$ for certain vertices, $u$, that are adjacent to $v$ in $S$.
 If the vertices of $S$ are moved a small enough distance, these determinants remain positive. Thus $S'$
is simplicial convex. Now we explain the determinant condition.

 Let $K$ be the union of the simplices in $S$ that contain $v$. 
 Suppose $\sigma$
 is a simplex in $K$ of dimension $(n-1)$.  Then $\sigma$ is contained in a hyperplane $H$.
 If $K$ is  simplicial convex, then all of $K$ is on one side of $H$.

 Conversely, $S$ is simplicial convex at $v$ if this condition holds
 for each such $\sigma$. Let $\{v_0,\cdots,v_{n-1}\}$ be the vertices of $\sigma$.
 Let $f(u)=\det(v_0-u:v_1-u:\cdots:v_{n-1}-u)$. The condition for simplicial convexity is
 that the sign of $f(u)$ is constant as $u$ ranges over all vertices in $K$ that are connected by an edge to $\sigma$
and that are not in $\sigma$. This ensures the simplices adjacent to $\sigma$ in $K$ all lie on the same side
of the hyperplane that contains $\sigma$.
 \end{proof}

\gap
The {\em convex hull, $\CH(Y)$,} of $Y\subset \RR^n$ is  the intersection of all the {\em closed} convex sets that contain $Y$.
Let $\pi_{\xi}:V_0\to\PP_+(V)$ be the projection $\pi_{\xi}(x)=[x]_+$.
The  next result can also be proven by using the existence of  an {\em affine sphere} \cite{MR437805, MR2743442}.  
\begin{theorem}\label{chulllemma} Suppose $\Ccal\subset\RR^{n+1}$ is a convex cone, $\Omega=\pi_{\xi}(\Ccal)\subset\RP_+^n$ and $M=\Omega/\Gamma$ is a properly convex closed manifold. 

Then there 
 is a convex proper submanifold $P\subset \Ccal$  with $t\cdot P\subset P$ for all $t\ge 1$,
  and $(\pi_{\xi}|\bdy P):\bdy P\to \Omega$ is a homeomorphism.
 Moreover, $P$ can be chosen such that $\bdy P/\Gamma\subset\Ccal/\Gamma$ is a compact, simplicial convex, simplicial hypersurface.
 \end{theorem}
\begin{proof} Let $\Dcal\subset\Ccal$  be the closed convex subset given by (\ref{SPdef}).
Then $N=\Dcal/\Gamma$ is a convex submanifold of $\Ccal/\Gamma$ with boundary
 a  strictly convex hypersurface. Let $H\subset \RR^{n+1}$ be an affine hyperplane
that separates $\Dcal$ into two components, such that the bounded component, $Y$,  projects to a subset $\pi(Y)\subset N$ that is contained 
in a small
ball in $N$.
Replace $N$ by $N\setminus \pi(Y)$. This produces a flat part in the boundary. 
This can be done finitely many times, to produce a
 submanifold $Q\subset N$
with simplicial boundary. The preimage of $Q$ in $\Ccal$ is the required submanifold $P$.  \end{proof}
    
 \begin{corollary}\label{flowconvex} A closed manifold $M$ is properly convex if and only if the tautological circle bundle $M\times S^1$ is radial convex.
 \end{corollary}
 \begin{proof} This follows from (\ref{chulllemma}) and (\ref{sufficient}).\end{proof}

\begin{corollary}[\cite{MR3500374}] A closed properly convex manifold $\Omega/\Gamma$ has a convex polyhedral fundamental domain
in $\Omega$.
  \end{corollary}
  \begin{proof} Use the notation in the proof of (\ref{chulllemma}).
  Let $H$ be a closed halfspace that contains $\Dcal$ and such that $(\bdy\Dcal)\cap\bdy H$ is a single point. 
  Let $Y=\cap(\gamma\cdot H)$ where the intersection is over all $\gamma\in\Gamma$. Then
  $\Dcal\subset Y\subset\Ccal\Omega$, and $Q=Y\cap\bdy H$ is a convex polytope.
  Moreover $\bdy Y=\Gamma\cdot Q$ is a locally finite union of images of $Q$.
 It follows that $\pi_{\xi}(Q)$
  is a convex polyhedral fundamental domain.
  \end{proof}

\begin{proof}[Proof of \Cref{open}]
Suppose $\rho_0\in \Rep_P(M)$. This has a tautological circle bundle $M\times S^1$. 
Choose a triangulation of $M\times S^1$ so that $M\times 1$ is a simplicial convex hypersurface.
If $\rho\in\Rep_P(M)$ is  close enough to $\rho_0$,  then  (\ref{ET}) gives a nearby projective structure  $M_{\rho}$ with
 holonomy $\rho$ so there is a nearby tautological circle bundle
$M_{\rho}\times S^1$ for this projective structure.
Moreover since the vertices have only been moved a small amount, (\ref{stabilitylemma}) implies that $M_{\rho}\times 1$ is simplicial convex in this structure, so $M_{\rho}\times S^1$ is radial convex.
 By (\ref{sufficient}) the  projective structure  $M_{\rho}$ is a properly convex structure. Thus $\rho\in\Rep_P(M)$.
\end{proof}

\section{Closed}

Here is an outline of the proof of the Closed Theorem.
Suppose $\rho_k$ is a sequence of holonomies of properly convex real projective structures on $M$,
so that $M\cong \Omega_k/\rho_k(\pi_1M)$ with $\Omega_k$ properly convex.
Suppose the holonomies converge pointwise to
 $\lim\rho_k=\rho_{\infty}$. If $M$ is strictly convex, a special case of Chuckrow's theorem (\ref{discfaith}) implies $\rho_{\infty}$
 is discrete and faithful; in general $\rho_{\infty}$ is neither.
 After taking a subsequence we may assume $\Omega_{\infty}=\lim\Omega_k\subset\RPn$ exists.
 If $\Omega_{\infty}$ is properly convex,  then $\Omega_{\infty}/\rho_{\infty}(\pi_1M)$ is a properly convex structure on $M$.
But $\Omega_{\infty}$
might have smaller dimension, or it might not be properly convex.  We describe this by saying {\em the domain
has degenerated}.  

An example of this is provided by a sequence of properly convex projective structures on a torus
given by $\Delta/\Gamma_n$, where $\Delta$ is the interior of a triangle and $\Gamma_n\cong \ZZ^2$ is a discrete group
of diagonal matrices (using a basis given by the vertices of $\Delta$)
generated by $\alpha_n$ and $\beta_n$ and $\alpha_n,\beta_n\to\Id$. There are conjugates of these groups that converge to a 
discrete group acting on the  plane, and the quotient is a Euclidean torus that is not properly convex.

The {\em box estimate} (\ref{boxestimate}) implies that one may replace the original sequence $\rho_k$
by conjugates $\rho'_k$ which preserve domains $\Omega'_k$ and  $\lim\rho_k'=\sigma$,
and 
$\lim \Omega'_k=\Omega$ is properly convex.
  Then $N=\Omega/\sigma(\pi_1M)$ is a properly convex manifold
 homotopy equivalent to $M$. Hence $N$ is closed  and $\pi_1M\cong\pi_1N$. Also $N$ is strictly convex
 by (\ref{stricthomeo}). Finally, $\sigma$ is irreducible by (\ref{irreducible}), which implies that, in fact,
 the original domains $\Omega_k$ did not degenerate. Thus $\rho_{\infty}$ is the holonomy
 of a strictly convex structure on $M$.  
 
 \medskip
 
 We supply the missing details, starting with the algebraic preliminaries.
  
\begin{theorem}[Irreducible]\label{irreducible} Suppose $M=\Omega/\Gamma$ is a strictly convex closed manifold and $\dim M\ge 2$.
Then  $\Gamma$ does not preserve any proper projective subspace. 
 \end{theorem} 
\begin{proof}  Suppose $\RR^{n+1}=U\oplus V$ with $U\ne 0\ne V$ and $\Gamma\subset\GL(n+1,\RR)$ preserves $U$. Then
$\Gamma$ preserves $Y=\cl\Omega\cap\PP_+ U$. 
Now  $Y\cap \Omega= \emptyset$ by (\ref{noinvt}).
Suppose $Y=\emptyset$. By the hyperplane separation theorem, there is $\phi\in (\RR^{n+1})^*$ such that $\ker \phi$ contains
$U$, and $\phi(\cl\Omega)>0$. 
Thus $[\phi]\in (\PP_+ U^0)\cap\Omega^*$.
By (\ref{dualcpct}) the dual manifold $M^*=\Omega^*/\Gamma$ is compact and strictly convex, and the dual action of $\Gamma$ preserves $U^0$, 
which  contradicts (\ref{noinvt}). Thus $Y\ne\emptyset$, so
 $Y\subset\Fr\Omega$.
 Since $M$ is strictly convex, $Y=c$ is a single point that is fixed by $\Gamma$. 
 
 Thus every non-trivial element of $\Gamma$ has an axis
 with one endpoint at $c$. 
  If $\ell$ and $\ell'$ are the axes of $\gamma,\gamma'\in\Gamma$,  then they limit on $c$. Now $\Omega$ is $C^1$ at $c$ by (\ref{C1})
 so $d_{\Omega}(p,\ell')\to0$ as the point $p$ on $\ell$ approaches $c$. Let $\pi:\Omega\to M$
 be the projection.  Then $C=\pi(\ell)$ and $C'=\pi(\ell')$ are closed geodesics in $M$
  that become arbitrarily close to each other. It follows that $C=C'$, so $\ell=\ell'$. Thus $\Gamma$
 preserves $\ell$ which contradicts (\ref{noinvt}) unless $\Omega=\ell$, but then $\dim M=1$.\end{proof}
 
The following is a special case of Chuckrow's theorem (see \cite{CHU} or \cite{KAP}(8.4)).
\begin{lemma}\label{discfaith} If $M$ is a closed, strictly convex, projective manifold
and $\dim M\ge 2$, then  the closure of $\Rep_S(M)$ in $\Rep(M)$ consists of discrete and faithful representations.
\end{lemma}
\begin{proof} Let $d$ be the metric on $G=\GL(m+1,\RR)$  given by $d(g,h)=\max|(g-h)_{ij}|$ and set $\|g\|=d(I,g)$.
Let $\Wcal\subset\pi_1M$ be a finite generating set. 
Suppose the sequence $\rho_n\in \Rep_S(M)$ converges to $\rho_{\infty}\in \Rep(M)$.
Then there is a compact set $K\subset G$ such that $\rho_n(\Wcal)\subset K$
for all $n$. 

The  map $\theta:G\times G\to G$ given by $\theta(g,h)=[[g,h],h]$  has zero derivative on $G\times\Id$.
Thus there is a neighborhood $U\subset G$ of the identity such that if $k\in K$ and $u\in U$,  then 
$\|\; [[k,u],u]\;\|\le \|\;u\;\|/2$.

Since $\rho=\rho_n$ is discrete, there is $1\ne\alpha\in\pi_1M$ that minimizes $\|\rho(\alpha)\|$.
Suppose that $\rho(\alpha)\in U$. If $\beta\in \Wcal$, then $\rho(\beta)\in K$, so $\|\;\rho[[\beta,\alpha],\alpha]\;\|\le \|\;\rho \alpha\;\|/2$. By minimality,
$\rho[[\beta,\alpha],\alpha]=\Id$, and
since $\rho$ is injective,
$[[\beta,\alpha],\alpha]=\Id$.  
 By (\ref{nilpotent}) $\alpha$ and $\beta$ commute.
Since $\Wcal$ is a generating set, $\alpha$ is central in $\pi_1M$. Thus the entire group preserves the axis of $\alpha$.
This contradicts (\ref{noinvt}) because $\dim M \ge 2$.
Thus $\rho(\alpha)\notin U.$ Since $U$ is an open neighbourhood of the identity, it contains an open metric ball $U'$, and since $\rho(\alpha)$ is of minimal norm, we have $\rho(\gamma)\notin U'$ for all $1\neq \gamma \in \pi_1M.$ This implies that $\rho_{\infty}$ is discrete and faithful.
\end{proof}

Let $\boxset=\prod_{i=1}^n[-1,1]\subset\RR^n\subset\RP^n$. For each $K>0$, the set $K\cdot \boxset=\prod_{i=1}^n[-K,K]$ is called a {\em box}.

\begin{lemma}[Box estimate]\label{boxestimate} If $A=(A_{ij})\in \GL(n+1,{\mathbb R})$
 and $K\ge1$ and $[A](\boxset)\subset K.\boxset$,  then $$|A_{ij}|\le 2K\cdot|A_{n+1,n+1}|$$
\end{lemma}
\begin{proof} Set $\alpha=A_{n+1,n+1}$. Using the standard basis we have
$$[x_1e_1+\cdots +x_ne_n+e_{n+1}]=[x_1:x_2:\cdots:x_n:1]=(x_1,\cdots,x_n)\in\boxset\quad \Leftrightarrow\quad  \max_i |x_i|\le 1$$ 
First consider the entries $A_{i,n+1}$ in the last column of $A$. Since $[e_{n+1}]=0\in \boxset$ we have 
$$[Ae_{n+1}]=[A_{1,n+1}e_1+A_{2,n+1}e_2+\cdots A_{n,n+1}e_n+\alpha e_{n+1}]\in K\cdot\boxset$$ 
It follows that $|A_{i,n+1}/\alpha|\le K$. This establishes the bound
when $j=n+1$ and $i\le n$.

Next consider the entries $A_{n+1,j}$ in the bottom row with $j\le n$. Observe that 
$$p=[t e_j+e_{n+1}]\in\boxset\quad\Leftrightarrow\quad |\;t\;|\le 1$$
Then $[A]p=[A(te_j+e_{n+1})]\in K\cdot \boxset$. This is in ${\mathbb R}^n$ so the $e_{n+1}$ component
is not zero. Hence $t A_{n+1,j}+\alpha\ne 0$ whenever $|\;t\;|\le 1$ and it follows that $|A_{n+1,j}| <|\alpha|$.
Since $K\ge 1$ the required bound follows when $i=n+1$ and $j\le n$.

The remaining entries are $1\le i,j\le n$. Since  $p=[p_1:\cdots:p_{n+1}]=[A(te_j+e_{n+1})]\in K\cdot \boxset$ it follows that
$$|t|\le1\quad\Rightarrow\qquad \left|\frac{p_i}{p_{n+1}}\right|=\left|\frac{A_{i,n+1}+t A_{i,j}}{\alpha + t A_{n+1,j}}\right| \le K$$
 For all $|t|\le1$ the denominator is not zero hence $|A_{n+1,j}|<|\alpha|$. It follows that
$$|\alpha + t A_{n+1,j}|\le 2|\alpha|$$
Thus 
$$|t|\le1\quad\Rightarrow\qquad\left|A_{i,n+1}+t A_{i,j}\right| \le 2K\cdot|\alpha|$$
We may choose the sign of $t=\pm1$ so that $A_{i,n+1}$ and $t A_{i,j}$ have the same sign. Then
$$ \left| A_{i,j}\right|\le \left|A_{i,n+1}+t A_{i,j}\right| $$
This gives the result
$ \left| A_{i,j}\right|\le  2|\alpha|\cdot K$ in this remaining case.
\end{proof}

If $\Omega\subset\RR^n$ has finite positive Lebesgue measure and the {\em centroid} of $\Omega$ is
 $\mu(\Omega)=0$, then
$$Q_{_{\Omega}}(y)=\int_K\left(\|x\|^2\|y\|^2-\langle x,y\rangle^2\right) \dvol_x$$  
 is a positive definite quadratic form on  $\RR^n$ called the {\em inertia tensor}.

\begin{lemma}[Uniform estimate]\label{Klemma} For each dimension $n$ there is $K=K(n)>1$ such that if $\Omega\subset{\mathbb R}^n$ is an open bounded convex
set with inertia tensor $Q_{_{\Omega}}=x_1^2+\cdots+x_n^2$ and centroid at the origin, then $K^{-1}\boxset\subset\Omega\subset K\cdot\boxset$. Moreover, if $A\in\GL(\Omega)$, then $|A_{ij}|\le K\cdot|A_{n+1,n+1}|$.
\end{lemma}
\begin{proof} The first conclusion follows from the theorem of Fritz John~\cite{Fritz}, see also  \cite{ball}.
Let $D\in\GL(n+1,\RR)$ be the diagonal matrix $\Diag(K,\cdots,K,1)$,  then $\boxset\subset D(\Omega)\subset K^2\boxset$.
Set $A'=D\cdot A\cdot D^{-1}$,  then $A'\in\GL(D(\Omega))$, thus $|A'_{ij}|\le 2K\cdot|A'_{n+1,n+1}|$ by  (\ref{boxestimate}).
Now $|A'_{n+1,n+1}|=|A_{n+1,n+1}|$ and $|A_{i,j}|\le K |A'_{i,j}|$ thus
$|A_{ij}|\le 2K^3\cdot|A_{n+1,n+1}|$. The result now holds using the constant $2K^3$.
  \end{proof}

\begin{proof}[Proof of \Cref{closed}]
Suppose $\rho\in\Rep(M)$ is the limit of the sequence $\rho_k\in\Rep_s(M)$.
Let $\Omega_k\subset\RP^n$ be the properly convex open set preserved by $\Gamma_k=\rho_k(\pi_1M)$, then $M\cong M_k=\Omega_k/\Gamma_k$.
Choose
an affine patch ${\mathbb R}^n\subset{\mathbb R}P^n$. Then by (\ref{centerexists}) there is $\alpha_k\in \PO(n+1)$
such that $\alpha_k(\Omega_k)\subset{\mathbb R}^n$ and has center $0\in {\mathbb R}^n$. We may choose $\alpha_k$ so that the
interia tensor $Q_k=Q_{\alpha_k(\Omega_k)}$ is diagonal in the standard coordinates on ${\mathbb R}^n$, and the entries on the main
diagonal of $Q_k$ are non-increasing going down the diagonal.
Since $\PO(n+1)$ is compact, after subsequencing
we may assume the conjugates of $\rho_k$ by $\alpha_k$ converge. 
We now replace the original sequence of representations
and domains by this new sequence.

Let $K=K(n)$ be given by (\ref{Klemma}). There is a unique positive diagonal matrix $D_k$
such that $Q_k=D_k^{-2}$. Set $\Omega'_k=D_k\Omega_k$, then $Q_{\Omega'_k}=x_1^2+\cdots+x_n^2$.
By (\ref{Klemma}), there is $K>1$ depending only on $n$, such that $$K^{-1}\cdot\boxset_1\subset \Omega'_k\subset K\cdot\boxset_1$$
Given $g\in\pi_1M$, then $A=A(k,g)=\rho_k(g)\in \SL(n+1,{\mathbb R})$ preserves $\Omega_k$.
 The matrix \[B=B(k,g)=D_k A(k,g) D_k^{-1}\] preserves $\Omega'_k$. By (\ref{Klemma})  $$\forall i,j\quad |B_{i,j}|\le K\cdot |B_{n+1,n+1}|$$ Since $D_k$ is diagonal it follows that $$B_{n+1,n+1}=A_{n+1,n+1}$$
Now $A_{n+1,n+1}=A(k,g)_{n+1,n+1}$ converges as $k\to\infty$ for each $g$. Hence the entries of $B(k,g)$ are uniformly bounded for fixed
$g$ as $k\to\infty$.
 Thus we may pass to a subsequence where $B(k,g)=D_k\rho_k(g)D_k^{-1}$ converges
for every $g\in\pi_1M$, and this gives a limiting representation $\sigma=\lim D_k\rho_kD_k^{-1}$. 

The space consisting of properly convex open sets $\Omega$
with $K^{-1}\cdot \Bcal_1\subset\Omega\subset K\cdot\Bcal_1$ is compact. Therefore there is a subsequence
 so that
$\Omega=\lim\Omega_k'$ exists. 
 Then $K^{-1}\cdot\boxset_1\subset \Omega\subset K\cdot\boxset_1$, hence
$\Omega$ is open and properly convex. Then $\sigma$ is discrete and faithful by (\ref{discfaith}), and $\Gamma=\sigma(\pi_1M)$
preserves $\Omega$,
so $N=\Omega/\Gamma$ is a properly convex manifold. Since $M$ is closed and $\pi_1M\cong \pi_1N$, then by (\ref{htpyequiv})
 $N$ is also closed. Since $M$ is strictly convex, (\ref{stricthomeo}) implies
 $N$ is
 strictly convex.

Since $\Rep_P(N)$ is open, for $D_k\rho_kD_k^{-1}$ close enough to $\sigma$, there is a properly convex open set $\Omega''_k\subset\RP^n$
that is preserved by $\Gamma_k$ and $N_k=\Omega''_k/\Gamma_k$
 is homeomorphic to $N$. By (\ref{stricthomeo}) $N_k$ is strictly convex. By (\ref{unique}) $\Omega''_k=\Omega_k$ so $N\cong N_k=M_k\cong M$. Thus $\sigma\in\Rep_S(M)$. 
 
 If $D_k$ does not remain bounded, since the entries are non-increasing going down the diagonal,  and $\rho_k(g)$ remains bounded,
 $\sigma(g)$ is block upper triangular  and therefore $\sigma$ is reducible. Since $N$ is strictly convex,  (\ref{irreducible}) implies $\sigma$ is irreducible,
therefore $D_k$ stays bounded. Hence we may subsequence the $D_k$ so they converge to $D$. Then $\rho=D^{-1}\sigma D$, so $\rho\in\Rep_S(M)$. 
\end{proof}

\small
\bibliographystyle{amsalpha}
\bibliography{Koszulrefs.bib}

\end{document}